\documentclass[10pt,a4paper,final]{amsart}

\usepackage{graphicx}%
\usepackage{multirow}%
\usepackage{amsmath,amssymb,amsfonts}%
\usepackage{amsthm}%
\usepackage{mathrsfs}%
\usepackage{stmaryrd}
\usepackage[title]{appendix}%
\usepackage{xcolor}%
\usepackage{textcomp}%
\usepackage{manyfoot}%
\usepackage{booktabs}%
\usepackage{algorithm}%
\usepackage{algorithmicx}%
\usepackage{algpseudocode}%
\usepackage{listings}%
\usepackage{hyperref}

\newtheorem{theorem}{Theorem}
\newtheorem{proposition}[theorem]{Proposition}%
\newtheorem{lemma}[theorem]{Lemma}
\newtheorem{corollary}[theorem]{Corollary}
\newtheorem{example}{Example}%
\newtheorem{remark}{Remark}%

\newtheorem{definition}{Definition}%

\begin{document}

\title[Asymptotic Dynamics on Character Varieties over Finite Fields]{Asymptotic Dynamics on Character Varieties over Finite Fields}


\author{Cigole Thomas}
\address{Department of Mathematical Sciences, Colorado State University}
\email{cigole.thomas@colostate.edu}







\begin{abstract}
In this paper, we prove the lack of asymptotic transitivity of the outer automorphism group action of $\mathbb{Z}^r$ on $\mathrm{SL}_n(\mathbb{F}_q)$-character varieties of $\mathbb{Z}^r$ for $n=2,3$ and $r\geq 2$. Along the way, we stratify the character varieties and compute the $E$-polynomial, also known as the Hodge-Deligne polynomial or Serre polynomial, of these character varieties.
\end{abstract}





\maketitle
\section*{Acknowledgements} The author expresses gratitude to Sean Lawton for suggesting the problem, offering guidance, engaging in productive discussions, and providing highly valuable suggestions. The author also extends thanks to Rachel Pries for dedicating time to read the manuscript and provide feedback. Additionally, the author also thanks Jeff Achter for various discussions and suggestions. The author is grateful to George Mason University for funding and resources while carrying out the research.

\section{Introduction}\label{sec1}
Let $\Gamma$ be a finitely presented group and $G=\mathrm{SL}_n(\mathbb{C})$. The $G$\textit{-character variety} of $\Gamma$, $\mathfrak{X}_\Gamma(G)$,  is the space of equivalence classes of group homomorphisms from $\Gamma$ to $G$, where two homomorphisms are equivalent if their conjugation orbit closures intersect (see \ref{CV}).  The outer automorphism group of $\Gamma$ acts on  $\mathfrak{X}_\Gamma(G)$ by precomposition (see \ref{eqn1}). A detailed discussion is provided in the next section.

In this paper, we study the dynamics of this action on the finite field points of the $\mathrm{SL}_n$-character variety of $\mathbb{Z}^r$ for $n=2,3$. The measure-theoretic dynamics of this action over fields of characteristic zero have been explored in specific cases for particular groups and varieties. 
 We aim to investigate whether there is an analogue of these results in the finite field case.  To this end, we introduce the concept of asymptotic transitivity (Definition \ref{AT}), and determine whether the action is asymptotically transitive in certain cases.  We discuss known results on ergodicity that motivate our findings in Section \ref{sec2}. The outline of the paper is as follows: In Section \ref{sec2}, we present some preliminaries and provide historical context for the research. In Section \ref{sec3}, we stratify the $\mathbb{F}_q$-points of the $\mathrm{SL}_3$-character variety of $\mathbb{Z}^r$  based on the stabilizer type under the conjugation action, where $\mathbb{F}_q$ is a finite field of order $q$. Section \ref{sec4} focuses on classifying and counting the number of degree $3$ monic characteristic polynomials with a constant term of $-1$.  Using the results from Section \ref{sec4}, Section \ref{sec5} is dedicated to counting the number of finite field points of the character variety in each stratum. Finally, we define asymptotic transitivity and explore this property in the context of the $\mathrm{Out}(\mathbb{Z}^r)$-action on the $\mathbb{F}_q$-points of the $\mathrm{SL}_n$-character variety of $\mathbb{Z}^r$ for $n=2,3$.
We say that the above action is asymptotically transitive if the ratio of the number of points in an orbit to the total number of points in the character variety approaches one  as $q \rightarrow \infty$ (see Definition \ref{AT}). 
In particular, we show that the action is not asymptotically transitive when $\Gamma=\mathbb{Z}^r$. We state the main theorem of the paper here. 
\begin{theorem}
	The action of $\mathrm{Out}(\mathbb{Z}^r)$ on the $\mathbb{F}_q$-points of the  $\mathrm{SL}_n$-character variety of $\mathbb{Z}^r$ is not asymptotically transitive for $n=2,3$ and $r\geq 2$. Furthermore, the asymptotic ratio of the orbits of elements in the character variety is bounded above by $\frac{1}{2}$. 
\end{theorem}
Another principal result presented in this paper is Theorem \ref{n3}, which provides a count of the number of $\mathbb{F}_q$-points in each stratum of the $\mathrm{SL}_3$-character variety of $\mathbb{Z}^r$. 

The result on asymptotic dynamics has the potential to be generalized in multiple directions, such as to $\mathrm{SL}_n$-character varieties of free groups $F_r$. Our aim is to draw parallels with topological dynamics on character varieties over fields of characteristic zero. The hope is that this can provide insight into the topological properties of  $\mathrm{SL}_n(\mathbb{C})$-character varieties  using arithmetic dynamics. 
\label{def}

\section{Preliminaries and Background}\label{sec2}
\subsection{Motivation and known results}
If $G$ is a complex affine reductive algebraic group and $\Gamma$ is a finitely presented group, then the set of $G$-representations of $\Gamma$ forms an algebraic set $\mathrm{Hom}(\Gamma,G)$. The group $G$ acts on the set of homomorphisms by conjugation. The $G$\textit{-character variety} of $\Gamma$ is the space of equivalence classes of group homomorphisms from $\Gamma$ to $G$, where two homomorphisms are equivalent if their conjugation orbit closures intersect. The $G$-character variety of $\Gamma$ is the categorical quotient in the category of affine varieties denoted by 
\begin{equation} \label{CV}
	\mathfrak{X}_\Gamma(G):=\mathrm{Hom}(\Gamma, G) \sslash G.
\end{equation}
This categorical quotient is constructed using Geometric Invariant Theory (GIT).
When $G$ is an affine algebraic group defined over the integers, $\mathbb{Z}$, 
 the locus of finite field points of the $G$-character variety of $\Gamma$ is well-defined. The automorphism group $\mathrm{Aut}(\Gamma)$ \label{aut} acts naturally on the $G$-representation variety of $\Gamma$ as follows: 
\begin{eqnarray}\label{eqn1}
	\mathrm{Aut}(\Gamma)  & \circlearrowleft  &  \mathrm{Hom}(\Gamma,G) \nonumber \\
	\tau \cdot \rho & = & \rho \circ \tau^{-1}.
\end{eqnarray} This leads to an action of $\mathrm{Out}(\Gamma)$ on the character variety, $\mathfrak{X}_\Gamma(G)$. We explore the dynamics of this action on the finite field points of
$\mathfrak{X}_\Gamma(G)$ in the following sections.
Along the way, we also recover the $E$-polynomials of $\mathrm{SL}_3$-character varieties of $\mathbb{Z}^r$. See Section \ref{Epolynomial} or \cite{HRV08} for an exposition of $E$-polynomials. 
In \cite{LawtonEpoly}, the $E$-polynomials of $\mathfrak{X}_{\mathbb{Z}^r}(\mathrm{SL}_n(\mathbb{C}))$ are calculated for $n=2,3$ using complex geometry methods.  We provide an arithmetic proof for the case when $n=3$, using a method that counts the number of possible characteristic polynomials for matrices over finite fields. An advantage of this method is that it also provides a count of orbits in each stratum for $n=3$.

The measure-theoretic dynamics of this action has been explored in specific cases for particular groups and varieties.  If $G$ is compact and connected, 
Goldman \cite{Goldman} and Pickerell-Xia \cite{PX} showed that when $\Gamma$ is a closed surface group of genus $g\geq 2$, there exists a natural measure class such that the action is ergodic. For complex groups, there have been studies that showed that this action is not ergodic. In \cite{Minsky13}, when $G=\mathrm{SL}_2(\mathbb{C})$ and $\Gamma=F_r$, a free group of rank $r$, Minsky described a domain of discontinuity for the action of $\mathrm{Out}(F_r)$ on $\mathfrak{X}_{F_r}(G)$ implying that the outer automorphism group does not act ergodically on the set of characters with dense image. For a free group, $F_r$, of rank $r\geq 3$, the action of $\mathrm{Out}(F_r)$ on the $G$-character variety of $F_r$ is ergodic with respect to an invariant measure when $G$ is compact and connected \cite{Gelander2008} and  \cite{Goldman2005}. 
When $\Gamma$ is the fundamental group of a non-orientable surface and $G=\mathrm{SU}(2)$, there exists a measure class with respect to which the action is ergodic \cite{Palesi09}. 
In \cite{BL2021}, Burelle and Lawton proved that for a compact connected Lie group  $G$, the  $ \mathrm{Out}(\Gamma)$-action is ergodic on the connected component of the identity of the character variety if $\Gamma$ is nilpotent and $\mathrm{Aut}(\Gamma)$ has a hyperbolic element. 

We are interested in studying the analog of these dynamical systems in an arithmetic setting. 
When $G$ is a reductive group defined over $\mathbb{Z}$ , we can consider the $\mathbb{Z}/p\mathbb{Z}$ points. Then we retain the action as defined above, but there is no longer a natural geometric invariant measure defined on the variety. However, it is still interesting to look at how close the action is to being transitive. 
In this setting, a comparable problem has been explored by Bourgain, Gamburd, and Sarnak in \cite{Sarnak2}, where they studied the $\mathbb{Z}/p\mathbb{Z}$ points of the  Markoff equation given by
\[x_1^2+x_2^2+x_3^2-3x_1x_2x_3=0\]  denoted as $\mathbb{X}(\mathbb{Z}/p\mathbb{Z})$, which is related to the $\mathrm{SL}_2$-character variety of the one-holed torus. 
They were interested in the action of the group $\Gamma$ of affine integral morphisms of an affine $3$-space generated by the permutations of the coordinates and Vieta involutions. Their results yield a strong approximation property for the Markoff equation for most primes, which is comparable to asymptotic transitivity. \\
In \cite{chen2021}, Chen shows that for all but finitely many primes $p$, the group of Markoff automorphisms acts transitively on the nonzero $\mathbb{F}_p$-points of the Markoff equation. This result proves that the action is asymptotically transitive on this variety. 
Goldman's ergodicity theorem for the compact case extends to relative character varieties when the genus, $g\geq 2$. The results from  \cite{chen2021} give reasons to believe that the action on the relative character variety for a one-holed torus, when $G$ is $\mathrm{SL}_2(\mathbb{F}_q)$, is also asymptotically transitive.
Cerbu, Gunther, Magee, and Peilen considered a related problem in ~\cite{Yale}. A similar problem has also been the subject of a project at Mason Experimental Geometry Lab (MEGL) at GMU where some progress has been made in collecting experimental data for small primes \cite{MEGL2015}. 

If we consider the case where the character variety is not relative (i.e., without fixing the boundary values) and where $G$ is a compact Lie group, then for a free group of rank $r\geq 3$, the action is ergodic when $G$ is compact and connected \cite{Gelander2008}. However, if $G$ is not compact, then the free group action on the character variety is not ergodic \cite{Minsky13}. In \cite{GLX2021} and \cite{PX2000}, the action is shown to be ergodic for certain relative character varieties. 

Hausel and Rodriguez-Villegas introduced arithmetic methods, inspired by the Weil conjectures to compute the $E$-polynomials of $G$-character varieties when $G=\mathrm{GL}_n(\mathbb{C}), \mathrm{SL}_n(\mathbb{C})$ and $\mathrm{PGL}_n(\mathbb{C})$. In \cite{HRV08},  they use a theorem of N. Katz [Appendix, \cite{HRV08}] that allows the calculation of $E$-polynomials by counting the finite field points of these varieties, and they obtain the $E$-polynomials for $G=\mathrm{GL}_n(\mathbb{C})$ as a generating function. In \cite{Mereb}, Mereb uses similar methods to calculate the $E$-polynomial when $G= \mathrm{SL}_2(\mathbb{C})$ , as well as a generating function for the $\mathrm{SL}_n(\mathbb{C})$ case. 
The Hodge polynomials of  $\mathrm{SL}_2(\mathbb{C})$ character varieties for curves of small genus were computed in \cite{LVN} by stratifying the space of representations and using fibrations.  The $E$-polynomial of $\mathfrak{X}_{\mathbb{Z}^r}(\mathrm{SL}_2(\mathbb{C}))$ has been calculated in \cite{SCEpoly} using arithmetic methods. In \cite{PLM21}, Gonz\'{a}lez-Perito, Logares, and Mu\~{n}oz compute the $E$-polynomial of the $\mathrm{AGL}_1(\mathbb{C})$-character varieties of closed, oriented surfaces of genus $g \geq 1$, where $\mathrm{AGL}_1(\mathbb{C})$ refers to the group of $\mathbb{C}$-linear affine transformations of the complex line suing three distinct methods for their calculations: the geometric method, the arithmetic method, and the quantum method.  In \cite{PLM23}, they compute the motives of the representation variety of torus knots of type $(m,n)$ into lower-rank affine groups over $\mathbb{C}$.  In \cite{PLM22}, Gonz\'{a}lez-Perito and Mu\~{n}oz extend this work to arbitrary fields, considering $\mathrm{AGL}_r(\mathbf{k})$ for ranks $r = 1$ and $2$.

 In \cite{FNZ},  Florentino, Nozad, and Zamora gave explicit expressions for the $E$-polynomial of the $\mathrm{GL}_n$-character varieties combining the combinatorics of partitions with arithmetic methods.
The authors extend the stratification by polystable type to $\mathrm{SL}_n$-character varieties in \cite{FLORENTINOSerre} and compute the $E$-polynomial of each stratum for the free case.
In \cite{FlorentinoSilva}, Florentino and Silva compute explicit formulas for Hodge-Deligne and E-polynomials of $\mathrm{SL}_n$-character varieties of free abelian groups using combinatorics of partitions. 

\subsection{$E$-polynomials}\label{Epolynomial}
We now define the $E$-polynomials. The following discussion is taken from \cite{SCEpoly} and \cite{LVN}. For an affine variety $X$, we consider the singular cohomology $H^*(X;\mathbb{F})$ where $\mathbb{F}$ is a field of characteristic $0$. See \cite{Deligne1} and \cite{Deligne3} for details.  A \textit{pure Hodge structure of weight $k$} consists of a finite dimensional complex vector space $H$ with a real structure, and a decomposition
$ H = \bigoplus\limits_{k=p+q} H^{p,q}$ such that $H^{q,p} =\overline{H^{p,q}}$, where $\overline{H^{p,q}}$ denotes the complex conjugate on $H$. A Hodge structure of weight $k$ gives rise to a descending\textit{ Hodge filtration},
$ F^p = \bigoplus\limits_{s\geq p} H^{s,k-s}.$

A complex variety $X$ admits a mixed \textit{mixed Hodge structure}, which consists of an increasing weight filtration, 
$$0=W_{-1}\subset W_0 \subset \cdots \subset W_{2j} = H^j(X;\mathbb{Q})$$ and a decreasing Hodge filtration 
$$H^j(X;\mathbb{C}) =F^0 \supset \cdots \supset F^{m+1}=0 \text{ such that for all } 0 \leq p \leq l,$$
$$ \mathrm{Gr}^{W\otimes \mathbb{C}}_l:= W_l\otimes \mathbb{C}/ W_{l-1}\otimes \mathbb{C} =F^p(\mathrm{Gr}^{W\otimes \mathbb{C} }_l)\oplus \overline{F^{l-p+1}(\mathrm{Gr}^{W\otimes \mathbb{C}}_l )}$$ where 
$$ F^p(\mathrm{Gr}^{W\otimes\mathbb{C}}_l) =(F_p\cap W_l \otimes \mathbb{C} +W_{l-1}\otimes \mathbb{C})/W_{l-1}\otimes \mathbb{C}.$$
Then we define the mixed Hodge number for $H^j(X;\mathbb{C})$ as follows: 
\begin{eqnarray*}
	h^{p,q;j}(X) & :=& \dim_{\mathbb{C}}\mathrm{Gr}_p^F(\mathrm{Gr}_{p+q}^{W\otimes \mathbb{C} }H^j(X))\\
	& = & \dim_{\mathbb{C}}F^p \cap (W_{p+q}\otimes \mathbb{C}))/(F^{p+1}\cap W_{p+q}\otimes \mathbb{C} +W_{p+q-1}\otimes \mathbb{C}\cap F^p),
\end{eqnarray*}
using which we define the mixed Hodge polynomial 
$$ H(X;x,y,t):=\sum h^{p,q;j}(X)x^py^qt^j.$$
The same structure can also be obtained by considering cohomology with compact support in a similar fashion. This is denoted by $H_c(X;\mathbb{F}), h_c^{p,q;j}$ (mixed Hodge numbers) and $H_c(X;x,y;t)$(mixed Hodge polynomial). 
Then the  \textit{$E$-polynomial} is defined to be $$E(X; x, y) := H_c(X; x, y,-1).$$
The classical Euler characteristic, $\chi(X)$ can be recovered as $\chi(X)=E(X;1,1)$. For further details, see \cite{CPJM08}. 
The main result used in the computation of E-polynomials is a theorem of Katz \cite[Theorem 6.1.2.3]{HRV08} which states that if the number of points of a variety over every finite field $\mathbb{F}_q$ is a polynomial in $q$, then this polynomial agrees with the $E$-polynomial of the variety. 

We use the setup from \cite{HRV08} to state the theorem. Let $X$ be a variety over $\mathbb{C}$. A \textit{spreading out} of $X$ is a seperated scheme $\mathcal{X}$ over a finitely generated $\mathbb{Z}$ algebra with an embedding $\phi: R \hookrightarrow \mathbb{C}$ such that the extension of scalars satisfy $\mathcal{X}_\phi \cong X$. Then  $X$ is said to have \textit{polynomial count }if there is a polynomial $P_X(t) \in \mathbb{Z}[t]$ and a spreading out $\mathcal{X}$ such that for every homomorphism $\phi:  R \rightarrow \mathbb{F}_q$ to a finite field (for all but finitely many primes $p$ so $q=p^k$), the number of $\mathbb{F}_q$-points of the scheme $\mathcal{X}_\phi$ is
$$\# \mathcal{X}_\phi(\mathbb{F}_q)=P_X(q).$$
Then we have the following theorem. 
\begin{theorem}\cite[Theorem 2.1.8, Katz]{HRV08}\label{ThmKatz}
	Let $X$ be a variety over $\mathbb{C}$. Assume $X$ has polynomial count with count polynomial $P_X(t) \in \mathbb{Z}[t]$, then the $E$-polynomial of $X$ is given by:
	$$ E(X;x,y)=P_X(xy).$$
\end{theorem}
This tells us that when we count the number of solutions to equations defining the variety over $\mathbb{F}_q$, if it is a universal polynomial evaluated at $q$ for all but finitely many $q$, then this polynomial determines the $E$-polynomial of the variety. 
In the first half of the paper, we focus on counting the finite field points of the character varieties. 
\subsection{Points in character variety as diagonal tuples}  
Let  $G=\mathrm{SL}_n(\mathbb{F}_q$) where $\mathbb{F}_q$ is a finite field with $q=p^k$ elements for a prime, $p$. 
\begin{lemma}
	The set of homomorphisms, $\mathrm{Hom}(\mathbb{Z}^r,G)$ is in bijective correspondence with the set of pairwise commuting $r$ tuples of $G$, $$ \{(A_1,\ldots,A_r) \in G^r\ | \ A_iA_j=A_jA_i \ \text{ for } 1\leq i,j \leq r\}.$$
\end{lemma}
\begin{proof}
	Recall that $\mathbb{Z}^r=\{(\gamma_1,\ldots, \gamma_r) \ | \gamma_i\gamma_j\gamma_i^{-1}\gamma_j^{-1}\text{ for } 1\leq i,j \leq r\}$. The evaluation map $ev: \mathrm{Hom}(\mathbb{Z}^r,\mathrm{SL}_n(\mathbb{F}_q)) \rightarrow G^r$ defined by 
	$ ev(\rho)=(\rho(\gamma_1),\ldots,\rho(\gamma_r))$ is surjective on the set of pairwise commuting tuples in $G$. It is clearly an injection and thus, $ev$ gives the desired bijective correspondence. 
\end{proof}
\noindent The conjugation action of $\mathrm{SL}_n(\mathbb{F}_q)$ on $\mathrm{Hom}(\mathbb{Z}^r,\mathrm{SL}_n(\mathbb{F}_q))$ is defined as $$ A\cdot\rho = A\rho A^{-1}=(A\rho(\gamma_1) A^{-1},\ldots,A\rho(\gamma_r) A^{-1}).$$
Let $\mathrm{Hom}(\mathbb{Z}^r, \mathrm{SL}_n(\mathbb{F}_q))/ \mathrm{SL}_n(\mathbb{F}_q)$ denote the set of orbits under this action. We first identify the elements in the set $\mathrm{Hom}(\mathbb{Z}^r, \mathrm{SL}_n(\mathbb{F}_q))$ that correspond to closed conjugation orbits which parameterizes the character variety, $\mathfrak{X}_{\mathbb{Z}^r}(\mathrm{SL}_n(\mathbb{F}_q))$. 
In \cite{FL2014}, Florentino and Lawton proved that there is a homeomorphism between the space of polystable orbits (closed orbits) and the GIT quotient. 
By Proposition 3.1 in \cite{FL2014}, for a finitely generated abelian group $\Gamma$ and a complex reductive algebraic group $G$, the set of polystable points, $\mathrm{Hom}(\Gamma,G)^{ps}$ is equivalent to $$\mathrm{Hom}(\Gamma,G_{ss})=\{ \rho \in \mathrm{Hom}(\Gamma,G) \ | \ \rho(\gamma_i)\in G_{ss}, i=1,\ldots,r\}$$  where $G_{ss}$ denotes the semisimple points in $G$. When $\Gamma =\mathbb{Z}^r$, the polystable points are homomorphisms, $\rho$ such that $\rho(\gamma_i)$ are simultaneously diagonalizable for all $i$, $1\leq i \leq r$. 

Through the evaluation map, we identify each $[\rho]\in \mathrm{Hom}(\mathbb{Z}^r, \mathrm{SL}_n(\mathbb{F}_q))\sslash \mathrm{SL}_n(\mathbb{F}_q)$  that corresponds to a pairwise commuting $r$-tuple of diagonal matrices in $\mathrm{SL}_n(\mathbb{F}_q)$.
To count the number of points in the character variety $\mathfrak{X}_{\mathbb{Z}^r}(\mathrm{SL}_n(\mathbb{F}_q))$, it suffices to count the number of distinct $r$-tuples of pairwise commuting diagonal matrices in $\mathrm{SL}_n(\mathbb{F}_q)$ up to the Weyl group action. It is easy to count the total number of diagonal tuples but the Weyl group action is different based on the stabilizer of the tuple under the conjugation action. So, we first stratify the space based on the stabilizer type of matrices under conjugation. 

\section{Stratification}\label{sec3}
We first compute the stabilizer of a tuple under the conjugation action. Let $\mathbb{F}_q$ denotes the finite field with $q=p^k$ elements where $p$ is prime and $\overline{\mathbb{F}}_q$ denote its algebraic closure. Throughout this paper, all polynomials that we mention have coefficients over the base field $\mathbb{F}_q$ unless otherwise specified. As we aim to count the number of tuples of matrices that are simultaneously diagonalizable over $\overline{\mathbb{F}}_q$, we consider the conjugation action over $\overline{\mathbb{F}}_q$ instead of $\mathbb{F}_q$. 

Consider the conjugation action of $\mathrm{GL}_n(\overline{\mathbb{F}}_q)$ on $\mathrm{SL}_n(\overline{\mathbb{F}}_q)^r$. For $P\in \mathrm{GL}_n(\overline{\mathbb{F}}_q)$ and $(A_1,\ldots,A_r) \in \mathrm{SL}_n(\overline{\mathbb{F}}_q)$, 
$$ P \cdot (A_1,\ldots,A_r) = (PA_1P^{-1},\ldots,PA_rP^{-1})$$ 
is in $\mathrm{SL}_n(\overline{\mathbb{F}}_q)^r$. 
Note that if a tuple,  $(A_1,\ldots,A_r)\in \mathrm{SL}_n(\mathbb{F}_q)^r$ is diagonalizable over $\mathrm{GL}_n(\overline{\mathbb{F}}_q)$, then the tuple is diagonalizable over $\mathrm{SL}_n(\overline{\mathbb{F}}_q)$ as well. 
To see this, suppose there exists $P\in \mathrm{GL}_n(\overline{\mathbb{F}}_q)$ such that the tuple $(PA_1P^{-1},\ldots, PA_rP^{-1})$ is diagonal. By letting $Q=\frac{1}{\sqrt[n]{\det P}} \cdot P$, we have $\det(Q)=1$ and $(QA_1Q^{-1},\ldots,QA_rQ^{-1})$ is diagonal. Therefore, to count simultaneously upper triangulizable matrices, it suffices to consider the action of  $\mathrm{SL}_n(\overline{\mathbb{F}}_q)$ instead of $\mathrm{GL}_n(\overline{\mathbb{F}}_q)$.

Note that pairwise commuting tuples in $\mathrm{SL}_n(\mathbb{F}_q)$ are simutaneously upper triangularizable over  $\overline{\mathbb{F}}_q$. But only the simultaneously diagonalizable tuples have closed orbits thereby consituting points in the character variety. Therefore there is a bijective correspondence between the closed conjugation orbits in $\mathrm{Hom}(\mathbb{Z}^r,\mathrm{SL}_2(\mathbb{F}_q))$ and the set of diagonal tuples in $\mathrm{SL}_2(\overline{\mathbb{F}}_q)^r$ up to the Weyl group action.   

\begin{definition}
	Two tuples $(A_1,\ldots, A_r)$ and $(B_1, \ldots ,B_r)$ in $ \mathrm{SL}_n(\mathbb{F}_q)^r$ are said to have the same stabilizer type if  $|\mathrm{Stab}(A_1,\ldots A_r)|=|\mathrm{Stab}(B_1, \ldots B_r)|$. 
\end{definition} 
\begin{definition}\label{G_A}
	For $A\in \mathrm{SL}_n(\overline{\mathbb{F}}_q)$, define the centralizer of $A$ in  $\mathrm{SL}_n(\mathbb{F}_q)$ as ,
	$$ G_A :=\{ P \in \mathrm{GL}_n(\overline{\mathbb{F}}_q)  \ | \ PA=AP\}.$$
\end{definition}
\noindent
Note that $G_A$ is the stabilizer of the conjugacy action and we are often concerned with the stabilizer in $\mathrm{SL}_n(\mathbb{F}_q)$ i.e., $G_A \cap \mathrm{SL}_n(\mathbb{F}_q)$. 
The following lemma characterizes the stabilizer when $r=1$.
\begin{lemma}\label{lem3.2.2}
	Let $D =[d_{ii}]_{i=1}^n \in \mathrm{GL}_n(\mathbb{F})$ be a diagonal matrix with diagonal entries, $d_{ii}$ from a field, $\mathbb{F}$. Let $G_D \subseteq \mathrm{GL}(\mathbb{F})$ be the set of all matrices that commute with $D$. Then $$G_D= \left\{ [a_{ij} ]_{i,j=1}^n\ | \ a_{ij}=0 \text{ if } d_{ii} \neq d_{jj}  \text{ and } \ a_{ij}\in \mathbb{F} \text{ if } d_{ii} = d_{jj}  \right\}\cap \mathrm{GL}_n(\mathbb{F}).$$
\end{lemma}
\begin{proof}
	Let $D$ be diagonal and $A\in \mathrm{GL}_n(\mathbb{F})$ be such that $AD=DA$. 
	Compare the $ij$-th entry of $AD$ and $DA$ for $1\leq i,j\leq n$  such that $d_{ii} \neq d_{jj}$ to get 
	\begin{eqnarray*}
		[AD]_{ij} &=& \sum_{k=1}^n a_{ik}d_{kj}  = a_{ij}d_{jj} \text{ and }
		\relax[DA]_{ij}  =   \sum_{k=1}^n d_{ik}a_{kj}  = d_{ii}a_{ij}
	\end{eqnarray*}
	Note that  $AD=DA$ if and only if  $a_{ij}(d_{jj}-d_{ii}) =0$ which is true  only if $a_{ij}=0$ or if $d_{ii}=d_{jj}$.  If $d_{ii}=d_{jj}$, then $ a_{ij}(d_{jj}-d_{ii}) =0$ for any value of $a_{ij} \in \mathbb{F}$.
\end{proof}

\begin{remark}\label{cor322}
	Suppose $D\in \mathrm{SL}_n(\overline{\mathbb{F}}_q)$ is a  diagonal matrix such that all the diagonal entries are distinct. If $A \in \mathrm{SL}_n(\overline{\mathbb{F}}_q)$ such that $AD=DA$, then $A$ is diagonal.
\end{remark}
We can now compute the stabilizer of a tuple $(A_1,\ldots,A_r) \in \mathrm{SL}_n(\mathbb{F}_q)$ in terms of the centralizers of $A_i$. 
\begin{lemma}\label{lem4.2.10}
	Let $(A_1,\ldots,A_r) \in \mathrm{SL}_n(\overline{\mathbb{F}}_q)^{ r} $. Consider the conjugation action of $\mathrm{GL}_n(\overline{\mathbb{F}}_q)$ defined by $P \cdot (A_1,....,A_r) = (PA_1P^{-1},....,PA_rP^{-1})$. \\Then   $\mathrm{Stab}_{\mathrm{GL}_n(\overline{\mathbb{F}}_q)}((A_1,....,A_r))=G_{A_1} \cap G_{A_2} \cap \ldots \cap G_{A_r}.$
\end{lemma}
\begin{proof} Observe that
		\begin{eqnarray*}	\mathrm{Stab}_{\mathrm{GL}_n(\overline{\mathbb{F}}_q)}((A_1,..,A_r)) &=& \{P \in \mathrm{GL}_n(\overline{\mathbb{F}}_q)\ | \ P \cdot(A_1,..,A_r)=(A_1,..,A_r)\} \\
	& = & G_{A_1} \cap G_{A_2} \cap \cdots \cap G_{A_r}.
	\end{eqnarray*}
\end{proof}

Given a matrix group, $G$, let $\mathscr{D}(G^r)$ denote the set of all diagonal matrices in $G^r$. For every matrix $P\in \mathrm{GL}_n(\overline{\mathbb{F}}_q)$ that stabilizes a tuple $(A_1,\ldots,A_r)\in \mathrm{SL}_n(\mathbb{F}_q)^r$, there exists a matrix in $\mathrm{SL}_n(\overline{\mathbb{F}}_q)$ that is in $\mathrm{Stab}({(A_1,\ldots,A_r))}$. To see this, suppose there exists $P\in \mathrm{GL}_n(\overline{\mathbb{F}}_q)$ such that the tuple $(PA_1P^{-1},\ldots, PA_rP^{-1})$ is diagonal. By letting $Q=\frac{1}{\sqrt[n]{\det P}} \cdot P$, we have $\det(Q)=1$ and $(QA_1Q^{-1},\ldots,QA_rQ^{-1})$ is diagonal. 

We now consider the case when $n=3$. In particular, we look at the stabilizer subgroup in $\mathrm{SL}_3(\mathbb{F}_q)$ i.e, $\mathrm{Stab}{((A_1,\ldots,A_r)})_{\mathrm{SL}_3(\mathbb{F}_q)} = \mathrm{Stab}{((A_1,\ldots,A_r)}) \cap \mathrm{SL}_3(\mathbb{F}_q)$.  
\begin{lemma}
	Let $(A_1,\ldots, A_r) $ be a tuple in $ \mathrm{SL}_3(\mathbb{F}_q)^r$ that is simultaneously diagonalizable to $(D_1,\ldots, D_r) \in \mathscr{D}(\mathrm{SL}_3(\mathbb{F}_q)^r)$, the set of all $r$-tuples of diagonal matrices in $\mathrm{SL}_3(\mathbb{F}_q)^r$. Let $d_{i_k}$ denote the $k^{th}$ diagonal entry of the matrix $D_i$. Then the three different stabilizer types under the simultaneous conjugation action are as follows: 
	\begin{enumerate}
		\item[1.] If $D_i$ is scalar for all $1\leq i \leq r$, then $\mathrm{Stab}((A_1,\ldots, A_r) )_{\mathrm{SL}_3(\mathbb{F}_q)}=\mathrm{SL}_3(\mathbb{F}_q)$.
		\item[2.] If there exists at least one $D_i$ such that all the entries of $D_i$ are distinct, then, $|\mathrm{Stab}((A_1,\ldots, A_r) )_{\mathrm{SL}_3(\mathbb{F}_q)}|= |\mathscr{D}(\mathrm{SL}_3(\mathbb{F}_q))|$.
		\item[3.]  If there exists exactly one pair $s,t\in \{1,2,3\}$ such that $d_{i_s}=d_{i_t}$ for all $1\leq i \leq r$, that is, same two rows have repeated entries for all $D_i$, then 
		\begin{eqnarray*}|\mathrm{Stab}((A_1,\ldots, A_r))_{\mathrm{SL}_3(\mathbb{F}_q)}| &= & |\{B=[b_{xy}] \ | \ b_{xx},b_{i_si_t},b_{i_ti_s} \in \mathbb{F}_q \\
			& &\text{ and } b_ {xy} =0 \text { for all other entries} \} \cap \mathrm{SL}_3(\mathbb{F}_q)|.
		\end{eqnarray*}
		\item[4.]	If there exist distinct $i,j$ such that $D_i$ and $D_j$ have exactly two repeated entries in different rows, i.e., $d_{i_s}=d_{i_t}$ and $d_{j_x}=d_{j_y}$ but $\{i_s,i_t\} \neq \{j_x,j_y\}$, then $|\mathrm{Stab}((A_1,\ldots, A_r) )_{\mathrm{SL}_3(\mathbb{F}_q)}|= |\mathscr{D}(\mathrm{SL}_3(\mathbb{F}_q))|$. 
	\end{enumerate} 
\end{lemma}
\begin{proof}
	
	\begin{enumerate}
		\item[1.] Since the central elements, $D_i$ commute with all matrices in $\mathrm{SL}_3(\mathbb{F}_q)$ for all $1\leq i \leq r$, the result follows from Lemma \ref{lem3.2.2} and Lemma \ref{lem4.2.10}. 
		\item[2.] Suppose there exists $i$ such that $D_i$ has distinct diagonal entries. Then, by Remark \ref{cor322}, $G_{D_i} \cap \mathrm{SL}_3(\mathbb{F}_q)= \mathscr{D}(\mathrm{SL}_3(\mathbb{F}_q))$. For $j\neq i$, $G_{D_i} \cap G_{D_j} \cap \mathrm{SL}_3(\mathbb{F}_q)=\mathscr{D}(\mathrm{SL}_3(\mathbb{F}_q))$. 
		By Lemma \ref{lem4.2.10}, $\mathrm{Stab}((D_1,\ldots,D_r))_{\mathrm{SL}_3(\mathbb{F}_q)} = G_{D_1}\cap \ldots \cap G_{D_r}\cap{\mathrm{SL}_3(\mathbb{F}_q)} =\mathscr{D}(\mathrm{SL}_3(\mathbb{F}_q))$. 
		Finally, since $(A_1,\ldots,A_r)$ and $(D_1,\ldots,D_r)$ are in the same orbit under conjugation, the stabilizer subgroups are conjugate. Therefore,  $|\mathrm{Stab}((A_1,\ldots,A_r))_{\mathrm{SL}_3(\mathbb{F}_q)}| =|\mathrm{Stab}((D_1,\ldots,D_r))_{\mathrm{SL}_3(\mathbb{F}_q)}|$.
		\item[3.] Let $A_i$ be a non-central matrix with exactly two repeated eigenvalues. 
		Then $A_i$ is diagonalizable to $D_i$ with exactly two repeated diagonal entries, say $d_{i_s}=d_{i_t}$ for all $i$, where $1\leq i \leq r$. By Lemma \ref{lem3.2.2}, $BD_i=D_iB$ implies $B= [b_{xy}] $ where  $b_{xx},b_{i_ti_s},b_{i_si_t} \in \mathbb{F}_q$   and $  b_{xy} =0 $ otherwise.
		
		Let $A_j$ be such that $1\leq j \leq r$ and $j\neq i$, then there are two possibilities for $A_j$. In the first case, $A_j$ has exactly two repeated eigenvalues and will be diagonalizable to $D_j$ with exactly two repeated entries, $d_{j_s}=d_{j_t}$. Then $G_{D_j}=G_{D_i}$ by the same explanation above. This implies $G_{D_i}\cap G_{D_j}=G_{D_i}$.
		The second possibility is that $D_j$ is scalar. Then $G_{D_j}\cap \mathrm{SL}_3(\mathbb{F}_q)=\mathrm{SL}_3(\mathbb{F}_q)$. Consequently, $G_{D_i} \cap G_{D_j}\cap \mathrm{SL}_3(\mathbb{F}_q)=G_{D_i}\cap \mathrm{SL}_3(\mathbb{F}_q)$. Thus in both cases, 
		\begin{eqnarray*} 
				\mathrm{Stab}((D_1,\ldots, D_r))_{\mathrm{SL}_3(\mathbb{F}_q)}&=&G_{D_1}\cap \cdots \cap G_{D_r}\cap \mathrm{SL}_3(\mathbb{F}_q)  =  G_{D_i}\cap \mathrm{SL}_3(\mathbb{F}_q)\\
					&= &\{[b_{xy}]  |  b_{xx},b_{i_ti_s},b_{i_si_t} \in \mathbb{F}_q \text{ and }  b_{xy} =0 \text{ otherwise}\} \cap \mathrm{SL}_3(\mathbb{F}_q).
		\end{eqnarray*}
		Since $|\mathrm{Stab}((D_1,\ldots, D_r))_{\mathrm{SL}_3(\mathbb{F}_q)}| = 	|\mathrm{Stab}((A_1,\ldots, A_r))_{\mathrm{SL}_3(\mathbb{F}_q)} |$, the result follows. 
		\item[4. ] Suppose $A_i$ has exactly one eigenvalue of mutliplicity two. Without loss of generality, we can simultaneously permute the rows of $D_i$ so that  $D_i$ has repeated entries, $d_{11}=d_{22}$. Let $A_j$ be such that $D_j$ has repeated entries, $d_{22}=d_{33}$ (The case when $d_{11}=d_{33}$ can be proved in the same way). As in the proof of part 3, by Lemma \ref{lem3.2.2}, we have the following
		\begin{eqnarray*}
			\hspace{1cm}G_{D_i} \cap \mathrm{SL}_3(\mathbb{F}_q)&=& \left\{\left. \begin{pmatrix}
				a_{11} & a_{12} & 0\\	a_{21} & a_{22} & 0 \\	0 & 0 & a_{33}
			\end{pmatrix}\ \right| \ a_{kk}, a_{12}, a_{21}\in \mathbb{F}_q \right\} \\
			\hspace{1cm}	G_{D_j}\cap \mathrm{SL}_3(\mathbb{F}_q)&= &\left\{\left. \begin{pmatrix}
				a_{11} & 0 & 0\\	0 & a_{22} & a_{23} \\	0 & a_{32} & a_{33}
			\end{pmatrix}\ \right| \ a_{kk}, a_{23}, a_{32}\in \mathbb{F}_q  \right\}.
		\end{eqnarray*}
	
		Clearly,  $G_{D_i} \cap G_{D_j} \cap \mathrm{SL}_3(\mathbb{F}_q)$
		is the set of all diagonal matrices in $\mathrm{SL}_3(\mathbb{F}_q)$. Since $\mathscr{D}(\mathrm{SL}_3(\mathbb{F}_q))\subseteq G_{D_h}$ for  $1\leq h\leq r$, it follows that 
		$ G_{D_1}\cap G_{D_2}\cap  \cdots \cap G_{D_r} \cap \mathrm{SL}_3(\mathbb{F}_q)=\mathscr{D}(\mathrm{SL}_3(\mathbb{F}_q)).$ By the same argument used in part 2, $|\mathrm{Stab}((A_1,\ldots,A_r))_{\mathrm{SL}_3(\mathbb{F}_q)}| =|\mathrm{Stab}((D_1,\ldots,D_r))_{\mathrm{SL}_3(\mathbb{F}_q)}|$. 
		
	\end{enumerate}
\end{proof}
 We classify the diagonal tuples,  $\mathscr{D}(\mathrm{SL}_3(\overline{\mathbb{F}}_q)^r)$
based on stabilizer type under the conjugation action and the field where the eigenvalues exist. Let $D_i \in \mathscr{D}( \mathrm{SL}_3(\overline{\mathbb{F}}_q))$ be a diagonal matrix in $\mathrm{SL}_3(\overline{\mathbb{F}}_q)$. 
\begin{enumerate}
	\item  Reducible tuples: $\mathscr{D}_1^r(\mathrm{SL}_3(\mathbb{F}_q))  :=  \{(D_1,\ldots ,D_r) \ | \   \text{ there exists } i \text{ such that }  \\ D_i  \text{ has} \text{ distinct diagonal entries where } i \in \{1, \ldots r\} \}$
	\item$ \text{ Repeating tuples: } \mathscr{D}_2^r (\mathrm{SL}_3(\mathbb{F}_q)):=  \{(D_1,\ldots ,D_r) \ | \  \text{there exists  } i \text{ such that }  \\D_i \text{ has}  \text { exactly  two distinct diagonal entries where } i \in \{1, \ldots r\} \}$	
	\item $	\text{Central tuples: } \mathscr{D}_3^r (\mathrm{SL}_3(\mathbb{F}_q)):= \{(D_1,\ldots ,D_r) \ | \ D_i  \text { is central for all } i \in \{1, \ldots r\} \} $	
	\item $\text{ Irreducible tuples: }	\overline{\mathscr{D}}_3^r (\mathrm{SL}_3(\mathbb{F}_q)) :=  \{ (D_1,\ldots ,D_r) \ | \ \text{there exists } i  \in \{1,\ldots, r\} \\\text{ such that } \text {all entries of }D_i \in \overline{\mathbb{F}}_q\setminus \mathbb{F}_q \} $
	\item$ \text{ Partially Reducible tuples: }	\overline{\mathscr{D}}_2^r (\mathrm{SL}_3(\mathbb{F}_q)) :=  \{ (D_1,\ldots ,D_r) \ | \  \text{there exists } i  \in \{1,\ldots, r\} \text{ such that } \text {exactly two entries of }D_i \in \overline{\mathbb{F}}_q\setminus \mathbb{F}_q \} $
\end{enumerate}
We stratify the space $\mathrm{SL}_3(\mathbb{F}_q)^r$ based on the type of diagonal tuple that an element is diagonalizable to. 
\begin{definition}\label{Strata3}
	A tuple $(A_1,\ldots,A_r)\in \mathrm{SL}_3(\mathbb{F}_q)^r$ is said to be in $X$ stratum if there exists $P  \in \mathrm{GL}_3(\overline{\mathbb{F}}_q)$ such that $(PA_1P^{-1},\ldots,PA_rP^{-1})$ is an X tuple. Replace X with Reducible, Repeating, Central, Irreducible, or Partially Reducible to obtain the definition for the corresponding stratum. We denote the stratum $X$ using $\mathscr{T}_X$.
\end{definition}

 \section{Counting Characteristic Polynomials}\label{sec4}
  
Let $p(x)$ be a degree three polynomial with coefficients from $\mathbb{F}_q$. Then $p(x)$ belongs to one of the following three types. 
\begin{enumerate}
	\item[1.] \textit{Completely reducible over the base field} (all roots are in $\mathbb{F}_q$)\\
	If $p(x)$ is completely reducible over $\mathbb{F}_q$, then $p(x)$ has one of the following: three repeated roots, exactly two repeated roots or three distinct roots. We use the term completely reducible to indicate matrices with such a $p(x)$ as characteristic polynomial. 
	\item[2.] \textit{ Irreducible over the base field} (no roots in $\mathbb{F}_q$)\\
	Recall that all irreducible polynomials over finite fields are separable [\cite{Dumfoot}, Chapter 13, Proposition 37]\label{prop3.2.27}. Therefore, if $p(x)$ is irreducible over $\mathbb{F}_q$, then all the roots are distinct.
	\item[3.] \textit{Partially reducible over the base field} (only some roots are in the base field) \\
	If $p(x)$ is partially reducible, then all the roots are distinct and exactly one of them will be in $\mathbb{F}_q$.
\end{enumerate}
We will now count the number of polynomials of each type, utilizing the following lemma, which is easy to prove.
\begin{lemma}\label{lem3}
	The number of irreducible monic polynomials of degree two over $\mathbb{F}_q$ is $\frac{q^2-q}{2}.$ 
\end{lemma}
\begin{proposition}\label{thm3.4.3}
	Let $S$ be the set of monic degree three polynomials, $f(x) \in \mathbb{F}_q(x)$ with constant term $-1$ where $q=p^k$. 
	\begin{enumerate}
		\item[1.] The number of reducible polynomials in $S$ with three repeated roots is 
		$$\begin{cases}
			3  & \text{  if  }  p\equiv 1 \bmod 3 \ \ \text{  or  } \ \ p\equiv -1 \bmod 3 \text{ and  } k \text{ is even }\\
			1  &  \text{  if  }   p\equiv 0 \bmod 3 \ \ \text{ or  }  \ \ p\equiv -1 \bmod 3 \text{ and  } k \text { is odd}.
		\end{cases}$$
		\item[2.] The number of polynomials in $S$ with exactly two repeated roots is 
		$$\begin{cases}
			q-4 & \text{ if } p\equiv 1 \bmod 3 \ \ \text{ or } \ \ p\equiv -1 \bmod 3 \text{ and  } k \text{ is even }\\
			q-2 &  \text{ if } p\equiv 0 \bmod 3 \ \ \text{ or } \ \ p\equiv -1 \bmod 3 \text{ and  } k \text{ is odd}.
		\end{cases}$$
		\item[3.]  The number of polynomials in $S$ with three distinct roots is
		$$ \begin{cases}
			\ \frac{q^2-5q+10}{6} &  \text{ if } p\equiv 1 \bmod 3 \ \ \text{ or } \ \  p\equiv -1 \bmod 3 \text{ and } k \text { is even }\\
			\frac{(q-3)(q-2)}{6} &  \text { if } p\equiv 0 \bmod 3 \ \  \text { or } \ \ p\equiv -1 \bmod 3 \text { and } k  \text{ is odd }.
		\end{cases}$$
		\item[4.] The number of polynomials in $S$ with exactly one root in $\mathbb{F}_q$ is $\frac{q^2-q}{2}$.
		\item[5.] The number of irreducible polynomials in $S$ is
		$$ \begin{cases}
			\frac{q^2+q-2}{3} & \text{ if } p\equiv 1 \bmod 3 \ \ \text{ or } \ \  p\equiv -1 \bmod 3 \text{ and } k \text { is even }\\
			\frac{q^2+q}{3} &  \text { if } p\equiv 0 \bmod 3 \ \  \text { or } \ \ p\equiv -1 \bmod 3 \text { and } k  \text{ is odd}.
		\end{cases}$$
	\end{enumerate}
\end{proposition}
\begin{proof}
	Let $p(x)=x^3+sx^2+tx-1$. Since there are $q$ choices each for $s$ and $t$, there are $q^2$ polynomials of degree three with constant term one. We classify and count them below based on the multiplicity of roots. 
	\begin{enumerate}
		\item[1.] Observe that $p(x)=(x-a)(x-a)(x-a)$ is in $S$ if and only if $-a^3 =1$.
		The number of cube roots of unity is $\gcd(3,q-1)$. 	
		\item[2.] If $p(x)$ has exactly two roots, then $p(x)=(x-a)(x-a)(x-\frac{1}{a^2})$ where $a\neq 0$ and $\frac{1}{a^2}\neq a$. 
		Out of the $(q-1)$ choices for $a$, it suffices to discount the case when $a=\frac{1}{a^2}$. But $a=\frac{1}{a^2}$ if and only if $a^3=1$. Subtracting this count (from part 1) from $q-1$ yields the desired result. 
		\item[3.] Suppose $p(x)=(x-a)(x-b)(x-\frac{1}{ab})$ where $a\neq b\neq \frac{1}{ab}$. There are $(q-1)$ choices for $a$ and $(q-2)$ choices for $b$. But this  includes the cases where $a=\frac{1}{ab}$ and $b=\frac{1}{ab}$. Note that this is exactly twice the count from part two. In addition, since any permutation of $a,b$ and $\frac{1}{ab}$ results in the same polynomial $p(x)$, we divide by $3!$ to get the final count. Therefore, 
		$$ \begin{cases}
			((q-1)(q-2)-2(q-4))\frac{1}{6}= \frac{q^2-5q+10}{6} & \text{ if } p\equiv 1 \bmod 3  \\ & \text{ or }  p\equiv -1 \bmod 3 \text{ and } k \text { is even }\\
			((q-1)(q-2)-2(q-2))\frac{1}{6}=\frac{q^2-5q+6}{6}&  \text { if } p\equiv 0 \bmod 3   \\ &  \text { or }  p\equiv -1 \bmod 3 \text { and } k  \text{ is odd }.
		\end{cases}$$
		\item[4.] Suppose $p(x)$ has exactly one root in $\mathbb{F}_q$. Then $p(x)=(x^2+ax+b)(x-\frac{1}{b})$ where $x^2+ax+b$ is an irreducible degree two polynomial. By Lemma \ref{lem3}, there are exactly $\frac{q^2-q}{2}$ such polynomials. Thus, the result follows. 
		\item[5.] Since the total number of polynomials possible is $q^2$, to get the number of irreducible polynomials, we subtract the other cases from $q^2$. 
		When  $p\equiv 1 \bmod 3 $ or $p\equiv -1 \bmod 3 $ and $k$ is even, we have 
		$$ q^2-\left(\frac{q^2-5q+10}{6}\right)-(q-4)-3-\left(\frac{q^2-q}{2}\right)=\frac{q^2+q-2}{3} .$$
		When $p\equiv 0 \bmod 3$ or $p\equiv -1 \bmod 3 $ and $k$ is odd, the result follows similarly. 
	\end{enumerate}
\end{proof} 

\section{Counting Points}\label{sec5}
\subsection{Weyl Group Action}
Let $(D_1,\ldots,D_r)$ be a tuple of diagonal matrices in $\mathrm{SL}_3(\overline{\mathbb{F}}_q)^r$ and for $(A_1,\ldots,A_r) \in \mathrm{SL}_3(\mathbb{F}_q)^r$, suppose there exists $P\in \mathrm{GL}_3(\overline{\mathbb{F}}_q)$ such that $(PD_1P^{-1},\ldots,PD_rP^{-1})=(A_1,\ldots,A_r)$. The set of diagonal matrices is a maximal torus and the corresponding Weyl group acts on the diagonal tuples by simultaneously permuting the diagonal entries. 
Let
\begin{eqnarray*} 
	\mathscr{W} & = &  \left\{ 
\begin{pmatrix}
	1& 0 &0 \\
	0 & 1 & 0\\
	0 & 0 & 1\end{pmatrix},
\begin{pmatrix}
	0& -1 & 0\\
	-1 & 0& 0\\
	0 & 0 & -1 \end{pmatrix},
\begin{pmatrix}
	0& 0 &-1 \\
	0 & -1 & 0\\
	-1 & 0 & 0\end{pmatrix}, \right. \\ 
& & \left. \begin{pmatrix}
	-1& 0 &0 \\
	0 & 0 & -1\\
	0 & -1 & 0\end{pmatrix},
\begin{pmatrix}
	0& 0 &1 \\
	1 & 0 & 0\\
	0 & 1& 0\end{pmatrix},
\begin{pmatrix}
	0& 1 &0 \\
	0 & 0 & 1\\
	1 & 0 & 0\end{pmatrix}
\right\} 
\end{eqnarray*}
be the Weyl group of $\mathrm{SL}_3(\mathbb{F}_q)$ which acts on diagonal matrices by conjugation. 
Note that for $W_i\in \mathscr{W}$ and a diagonal matrix $D$, $W_iDW_i^{-1}$ is just a permutation of its entries. We want to compute the distinct orbits under the Weyl group action. We first classify the elements of $\mathscr{D}_3^r(\mathrm{SL}_3(\mathbb{F}_q))$ (basefield stratum) based on the number of distinct permutations possible under the action of $\mathscr{W}$. 
For $G=\mathrm{SL}_n(\mathbb{F}_q)$ and a stratum $X$, we use the notation $\mathscr{T}_X/G$ to denote the orbits under the conjugation action of $G$. Similarly, we use $\mathscr{T}_X\sslash G$ to denote the closed conjugation orbits or equivalently the orbits in $\mathscr{T}_X/G$ under Weyl group action. The following lemma computes the number of distinct conjugation orbits arising from a single diagonal tuple under the Weyl group action. 

\begin{lemma} \label{Rmk2}
	Let $(A_1,\ldots, A_r)\in \mathrm{SL}_3(\mathbb{F}_q)^r$ be diagonalizable to the tuple $(D_1,\ldots,D_r)$ in $ \mathrm{SL}_3(\mathbb{F}_q)^r$. Then the following are true. 
	\begin{enumerate}
		\item If there exists $i$ such that $D_i$ has distinct entries, then all the six permutations of $D_i$ and consequently that of $(D_1,\ldots,D_r) \in \mathrm{SL}_3(\mathbb{F}_q)$ are distinct.
		\item Suppose there exists $i,j$ such that $D_i$ and $D_j$ both have exactly two repeated entries but in different rows, i.e., $d_{i_s}=d_{i_t}$ and $d_{j_x}=d_{j_y}$ but $\{i_s,i_t\} \neq \{j_x,j_y\}$. Then, $(D_1,\ldots,D_r) $ has six distinct permutations. 
		\item If $D_i$ has two repeated entries at the same position, say $d_{i_s}=d_{i_t}$, and $D_i$ is not scalar for all $1\leq i \leq r$, then there are exactly three  distinct permutations of $(D_1,\ldots,D_r) \in \mathrm{SL}_3(\mathbb{F}_q).$ 
		\item All permutations of $(D_1,\ldots,D_r) \in \mathrm{SL}_3(\mathbb{F}_q)$ are identical if each $D_i$ is a scalar matrix for all $i$. 
	\end{enumerate}
	
\end{lemma}
\begin{proof} 
	\begin{enumerate}
		
		\item  Suppose  $D_i =\begin{pmatrix}
			d_1 &  0 & 0 \\0 & d_2 & 0\\
			0 & 0 & d_3
		\end{pmatrix}$  where $d_i\neq d_j$ for $i \neq j$. Then for $W \in \mathscr{W}$, $WD_iW^{-1} =D_i$ if and only if $W$ is the identity since any other $W$ permutes the rows of $D_i$. Consequently, $(WD_1W^{-1},\ldots,WD_rW^{-1})=(D_1,\ldots,D_r)$ if and only if $W$ is the identity matrix. 
		\item Let $D_i$ and $D_j$ be such that ${i_s,i_t}$ denotes the rows of $D_i$ with repeated entries, and ${j_x,j_y}$ denotes the rows of $D_j$ with repeated entries, where $\{i_s,i_t\} \neq \{j_x,j_y\}$ Since $\{i_s,i_t\}, \{j_x,j_y\} \subseteq \{1,2,3\}$, it follows that $\{i_s,i_t\}\cup \{j_x,j_y\} =\{1,2,3\}$. Let $W\neq I$. Then $W$ permutes at least two rows, say $u,v$. If $u,v\in \{i_s,i_t\} $ , i.e., $W$ does not permute any elements in $D_i$, then since $\{i_s,i_t\}\neq \{j_x,j_y\} $, either $u\notin \{j_x,j_y\} $ or $v \notin \{j_x,j_y\} $. Since two distinct entries of $D_j$ are permuted by $W$, it follows that $WD_jW^{-1} \neq D_j$. Thus,  $(WD_1W^{-1},\ldots,WD_rW^{-1})\neq(D_1,\ldots,D_r)$ and hence the tuple has six distinct permutations. 
		\item WLOG, assume each of the $D_i$ is of the form
		$D_i =\begin{pmatrix}
			d_i &  0 & 0 \\0 & d_i & 0\\
			0 & 0 & \frac{1}{d_i^2}
		\end{pmatrix}$ such that $d_i\neq \frac{1}{d_i^2}$. 
	It suffices to count the number of distinct ways to choose two rows to place the $d_i$s. This is given exactly by $\binom{3}{2}$. Therefore, there are exactly three matrices in $\mathscr{W}$ such that $WD_iW^{-1}\neq D_i$. Since all the $D_i$ have the same form or are scalar matrices, it follows that $(D_1,\ldots,D_r)$ has exactly three distinct simultaneous permutations by $\mathscr{W}$. 
	\item If $D_i$ is a scalar matrix for all $i$, then clearly $WD_iW^{-1}=D_i$ for all $i$ and $W\in \mathscr{W}$. 
	\end{enumerate}
	
\end{proof}
We now present some results that help identify the number of distinct permutations when the diagonal entries are elements of $\overline{\mathbb{F}}_q$.
\begin{lemma}\label{|G_A|}
	Let $\mathrm{SL}_n(\mathbb{F}_q)$ act on $\mathrm{SL}_n(\mathbb{F}_q)$ by conjugation. If $A$ is such that its characteristic polynomial, $\mathrm{char}(A)$, is irreducible over $\mathbb{F}_q$, then $G_A\cap \mathrm{SL}_n(\mathbb{F}_q)$, the stabilizer of $A$ in $\mathrm{SL}_n(\mathbb{F}_q)$ is of size $\frac{q^n-1}{q-1}$. 
\end{lemma}
\begin{proof}
	Let $A\in \mathrm{SL}_n(\mathbb{F}_q)$ where $n=2,3$. Let  $\mathrm{SL}_n(\mathbb{F}_q)$ act on $\mathrm{SL}_n(\mathbb{F}_q)$ by conjugation: 
	$ P \cdot A =PAP^{-1}.$
	Recall that matrices conjugate to $A$ have the same characteristic and minimal polynomial. Since $\mathrm{char}(A)$ is irreducible over $\mathbb{F}_q$, $\mathrm{char}(A)$ is the same as the minimal polynomial of $A$. 
	Consequently, $B\in \mathrm{SL}_n(\mathbb{F}_q)$ is in the orbit of $A$ if and only if $\mathrm{char}(B)=\mathrm{char}(A)$. 
	To calculate $|G_A\cap \mathrm{SL}_n(\mathbb{F}_q)|$, we begin by computing the order of the orbit and then apply the Orbit-Stabilizer theorem. We use the following result from \cite{RTH2014} to compute the size of $\mathrm{Orb}(A)$. Let $M_n$ denote the set of all matrices with entries in $\mathbb{F}_q$. \\	 
	 {\bf Theorem.} (see \cite[Theorem 1]{RTH2014}) 
	\textit{Let $f(x) \in \mathbb{F}_q(x)$ be an irreducible polynomial of degree $n$. Then,
		the number of matrices in $M_n$ with characteristic polynomial $f$ is}
	$ \prod\limits_{i=1}^{n-1}(q^n-q^i).$\\
	\noindent
	Therefore, $|\mathrm{Orb}(A)|=\prod\limits_{i=1}^{n-1}(q^n-q^i)$. And $|\mathrm{SL}_n(\mathbb{F}_q)|=\frac{1}{q-1}\prod\limits_{i=0}^{n-1}(q^n-q^i).$
	Then, by the Orbit-Stabilizer Theorem, for $n=2,3$
	$$|G_A\cap \mathrm{SL}_n(\mathbb{F}_q)|= |\mathrm{Stab}_{\mathrm{SL}_n(\mathbb{F}_q)}(A)|=\frac{|\mathrm{SL}_n(\mathbb{F}_q)|}{|\mathrm{Orb}(A)|}=\frac{\frac{\prod\limits_{i=0}^{n-1}(q^n-q^i)}{q-1}}{\prod\limits_{i=1}^{n-1}(q^n-q^i)}=\frac{q^n-1}{q-1}.$$
	Therefore, the result follows. 
\end{proof}
We omit the proof of the following well known result from linear algebra. 
\begin{lemma}\label{3.2.12}
	Let $\mathbb{F}$ be a field and let $V$ be a $\mathbb{F}$-vector space. Let $T : V\rightarrow V$ be a linear transformation. Then the characteristic polynomial of $T$ is irreducible over $\mathbb{F}$ if and only if $T$ has no non-trivial invariant subspaces. 
\end{lemma}

\noindent
For a matrix $A$, let $\mathrm{char}(A)$ \label{char}denote the characteristic polynomial of $A$.  
\begin{lemma}\label{lem4}
	Let $A,B$ be commuting matrices. If $\lambda$ is an eigenvalue of $A$, then $B$ preserves the $\lambda$-eigenspace of $A$. 
\end{lemma}
\begin{proof}
	Let $Au=\lambda u$. 
	Then, $A(Bu)=BAu=B\lambda u=\lambda B u.$
	Thus, $Bu\in \mathrm{Eig}_\lambda(A)$.
\end{proof}
\begin{lemma}\label{-D2}
	Let $A\in \mathrm{SL}_3(\mathbb{F}_q)$ and $\mathrm{char}(A)=(x-a)(x^2+sx+\frac{1}{a})$ be such that $(x^2+sx+\frac{1}{a})$ is irreducible over $\mathbb{F}_q$. Then if $B\in \mathrm{SL}_3(\mathbb{F}_q)$ commutes with $A$\, then $\mathrm{char}(B)=(x-b)(x^2+tx+\frac{1}{b})$  such that $(x^2+tx+\frac{1}{b})$ is irreducible over $\mathbb{F}_q$ or $\mathrm{char}(B)=(x-d)(x-c)(x-c)$ where $c,d\in \mathbb{F}_q$. 
\end{lemma}
\begin{proof}
	Let $A\in \mathrm{SL}_3(\mathbb{F}_q)$ be such that $\mathrm{char}(A)=(x-a)(x^2+sx+\frac{1}{a})$ where $x^2+sx+\frac{1}{a}$ is irreducible over $\mathbb{F}_q$. If $AB=BA$, then $B$ preserves the eigenspace, $W_a$ of $a$
	by Lemma \ref{lem4}. 
	Since $A\in \mathrm{SL}_3(\mathbb{F}_q)$ and $a\in \mathbb{F}_q$, note that $W_a\subset \mathbb{F}_q^3$.
	This implies that $B$ has an eigenvalue $c$ in $\mathbb{F}_q$ ie., $\mathrm{char}(B)=(x-c)f(x)$ where $f(x)$ is a degree two polynomial over $\mathbb{F}_q$.

	If $\mathrm{char}(B)=(x-c)(x-c_1)(x-c_2)$ where $c_1\neq c_2$ and $c_1,c_2\in \mathbb{F}_q$, then $A$ preserves the one-dimensional eigenspaces $W_{c_1}$ and $W_{c_2}$ by Lemma \ref{lem4}. This is a contradiction and hence $f(x)$ cannot have distinct roots over $\mathbb{F}_q$. 
\end{proof}
\begin{lemma}\label{irred}
	Let $A,B \in \mathrm{SL}_n(\mathbb{F}_q)$ such that $AB=BA$ and $\mathrm{char}(A)$ is irreducible over $\mathbb{F}_q$ for $n=2,3$. Then $B$ is either central or has an irreducible characteristic polynomial. 
\end{lemma}
\begin{proof}
	Suppose $B$ is not central. If $\mathrm{char}(B)$ is reducible, then it has a non-trivial divisor $x-a$ with multiplicity one such that $a\in \mathbb{F}_q$. Then the corresponding eigenspace, $W_a \subset \mathbb{F}_q^3$ is preserved by $A$, by Lemma \ref{lem4}.  This is a contradiction since $A$ has no non-trivial subspaces by Lemma \ref{3.2.12}. Therefore, the characteristic polynomial of $B$ is irreducible. 
\end{proof}

\begin{lemma}\label{commutativity}
	Let  $A\in \mathrm{GL}_n(\mathbb{F}_q)$  have an irreducible characteristic polynomial over $\mathbb{F}_q$. If $q(x)$ is an irreducible polynomial of degree $n$ over $\mathbb{F}_q$, then there exists $B\in \mathrm{GL}_n(\mathbb{F}_q)$ that commutes with $A$ such that the characteristic polynomial of $B$ is $q(x)$.
\end{lemma}
\begin{proof}
	Let $A\in \mathrm{GL}_n(\mathbb{F}_q)$ such that the characteristic polynomial, $\mathrm{char}(A)=p(x)\in \mathbb{F}_q(x)$ is irreducible over $\mathbb{F}_q$. Let $M_n(\mathbb{F}_q)$ denote the set of $n \times n$ matrices over $\mathbb{F}_q$. Define the map: 
	\begin{eqnarray*}
	 \tilde{\phi}: \mathbb{F}_q[x]  &\longrightarrow & M_n(\mathbb{F}_q) \\
	 f(x) &\longmapsto& f(A)
	 \end{eqnarray*}
	 Since $p(x)$ is irreducible, the minimal polynomial of $A$ is the same as the characteristic polynomial i.e., $p(A)=0$. Therefore,  $p(x) \in \ker(\tilde{\phi})$. 
	 Let $L:=\frac{\mathbb{F}_q(x)}{<p(x)>}$. Since $p(x)$ is irreducible and $M_n(\mathbb{F}_q)$ is finite dimensional over $\mathbb{F}_q$, we have an inclusion: 
	 $ \phi: L \hookrightarrow M_n(\mathbb{F}_q)$. 
	 Then, every $U \in \phi(L)$ commutes with $A$ since $U$ is a polynomial in $A$ and $A$ commutes with scalars and powers of $A$. 
	 Restricting the inclusion to the set of units in $L$ induces an inclusion,
	 $\phi: L^\times \hookrightarrow \mathrm{GL}_n(\mathbb{F}_q).$
	 
	 Let $c\in L$ have the minimal polynomial, $m(x)$ and let $C=\phi(c)$. Since $m(c)=0$, we have that $m(C)=0$ which implies that $m(x)$ divides the characteristic polynomial, $\mathrm{char}(C)$. If $m(x)$ is irreducible such that $\mathrm{deg}(m)=n=[L:\mathbb{F}_q]$, then $\mathrm{char}(C)=m(x)$. 
	 
	 Now suppose $q(x)$ is an irreducible monic polynomial of degree, $n$. 
	 Since $\mathbb{F}_q$ has a unique extension $L$ of degree $n$, if $b \in L$ is any root of $q(x)$, then $B=\phi(b)$ has characteristic polynomial $q(x)$. As noted before, then $B \in \phi(L)$ and hence commutes with $A$. 
\end{proof}

\begin{lemma}\label{lem3.2.26}
	Let $A\in \mathrm{SL}_3(\mathbb{F}_q)$ be such that $\mathrm{char}(A)$ is not completely reducible over $\mathbb{F}_q$  where $q=p^k$ for a prime $p$. If $A$ is diagonalizable to $D\in \mathrm{SL}_3(\overline{\mathbb{F}}_q)$, then all the entries of $D$ are distinct. Furthermore, let $(A_1,\ldots,A_r)$ be an element of the irreducible stratum such that $A_i=A$ for some $i$. Then $(A_1,\ldots,A_r)$ has exactly three distinct permutations if $\mathrm{char}(A)$ is irreducible.
\end{lemma}

\begin{proof}
	Let $A$ and $D$ be as defined in the statement of the lemma.  Then there are two possibilities for the characteristic polynomial, $\mathrm{char}(D)=p(x)$ :
	$(x-a)(x^2+bx+1/a)$ or $x^3+ax^2+bx-1$. 
	
	\begin{itemize}
		\item $p(x)=x^3+ax^2+bx-1$:  As discussed earlier in Section \ref{prop3.2.27}, $p(x)$ is separable. Therefore, if $\mathrm{char}(D)=x^3+ax^2+bx-1$, then $D$ has distinct entries.   
		\item If $\mathrm{char}(D)= (x-a)(x^2+bx+\frac{1}{a})$ such that $a\in \mathbb{F}_q$, then $(x^2+bx+\frac{1}{a})$ has two distinct roots say $d_1,d_2$. Since irreducible polynomials are separable over finite fields, $d_1\neq d_2$. Clearly, $d_1,d_2\neq a$ since $a\in \mathbb{F}_q$. Therefore, $D$ has distinct entries. 
	\end{itemize}
	Now suppose $\mathrm{char}(A)=p(x)=x^3+ax^2+bx-1$ is irreducible.  We proceed to show that a tuple containing $A$ has exactly three distinct permutations.

	 From Lemma \ref{|G_A|}, the size of $|G_A\cap \mathrm{SL}_3(\mathbb{F}_q)|$ is $q^2+q+1$. If $(A_1,\ldots,A_r)\in \mathrm{SL}_3(\mathbb{F}_q)^r$ is such that $A_i=A$ for some $A$, then $A_j\in G_A$ for $1\leq j \leq r$.  By Lemma \ref{irred}, any matrix $A_j$ that commutes with $A$ is either a central element or has an irreducible characteristic polynomial. Note that by Lemma \ref{commutativity}, if $q(x)$ is any irreducible monic degree three polynomial with constant term $-1$, then there exists a matrix $B\in \mathrm{SL}_3(\mathbb{F}_q)$ with $q(x)$ as the characteristic polynomial such that $B$ commutes with $A$ . Consequently, every such $q(x)$ appears as the characteristic polynomial of a matrix in $G_A$.   Let $d_3$ be the number of central elements, $m_3$ be the number of permutations of $A$ in $\mathrm{SL}_3(\mathbb{F}_q)$ and $l$ denote the number of irreducible degree three polynomials with determinant $-1$ from Proposition \ref{thm3.4.3}. 
	 Then, 
	 \begin{eqnarray*} m_3 &=& \frac{|G_A\cap \mathrm{SL}_3(\mathbb{F}_q)-d_3|}{l} \\
	 	&=&\begin{cases}  \frac{q^2+q+1-3}{\frac{q^2+q-2}{3}}=3\text{ if } p\equiv 1 \bmod 3  \text{ or }  p\equiv -1\bmod 3  \text{ and } k \text{ is even}  \\
		\frac{q^2+q+1-1}{\frac{q^2+q}{3}}=3; \text { if } p\equiv 0 \bmod 3   \text { or }  p\equiv -1 \bmod 3 \text { and } k  \text{ is odd.}
	\end{cases}
	\end{eqnarray*}
\end{proof}
Let $Z:=Z(\mathrm{GL}_2(\mathbb{F}_q))$ denotes the center of $\mathrm{GL}_2(\mathbb{F}_q)$, the set of scalar matrices with entries from $\mathbb{F}_q$. Let $\overline{\mathscr{D}}(\mathrm{GL}_2(\overline{\mathbb{F}}_q -\mathbb{F}_q))$ denote the set of $2 \times 2$ diagonal matrices with entries from $\overline{\mathbb{F}}_q-\mathbb{F}_q$. In the following lemma, we use $\overline{\mathscr{D}}_2$ to denote $\overline{\mathscr{D}}(\mathrm{GL}_2(\overline{\mathbb{F}}_q -\mathbb{F}_q))$ and $Z$ to denote the central matrices in $\mathrm{GL}_2(\mathbb{F}_q)$. 
 Recall the definition of partially reducible tuples: 
$\overline{\mathscr{D}}_2^r (\mathrm{SL}_3(\mathbb{F}_q)) :=  \{ (D_1,\ldots ,D_r) \ | \  \text{there exists } i  \in \{1,\ldots, r\} \text{ such that } \text {exactly two entries of }D_i \in \overline{\mathbb{F}}_q\setminus \mathbb{F}_q \} $ and the stratum $\mathscr{T}_{\overline{D}_{2}^r}(\mathrm{SL}_3(\mathbb{F}_q))$ as the set of tuples in $\mathrm{SL}_3(\mathbb{F}_q)$ that are simultaneously diagonalizable to a partially reducible tuple [see Definition \ref{Strata3}].
  In the following proof, we sometimes omit $\mathrm{SL}_3(\mathbb{F}_q)$ from the bracket of the stratum for convenience. 
\begin{lemma}\label{PRS} 
	There is a bijective correspondence between the orbits in $\mathscr{T}_{\overline{D}_{2}^r}(\mathrm{SL}_3(\mathbb{F}_q))$ and the set of commuting tuples in $(\overline{\mathscr{D}}_2\cup Z)^r - Z^r$ upto simultaneous permutation. 
\end{lemma}

\begin{proof}
	We want to prove that there is a bijection between the number of closed conjugation orbits in the partially reducible stratum of $\mathrm{SL}_3(\mathbb{F}_q)$ (orbits in $\mathscr{T}_{\overline{D}_{2}^r}(\mathrm{SL}_3(\mathbb{F}_q))$ and the irreducible diagonal $r$-tuples of $\mathrm{GL}_2(\overline{\mathbb{F}}_q)$ excluding the central tuples ($(\overline{\mathscr{D}}_2\cup Z)^r - Z^r$). 
	Let  $(A_1,\ldots,A_r) \in  \mathscr{T}_{\overline{D}_{2}^r}(\mathrm{SL}_3(\mathbb{F}_q))$. By definition of the stratum, for some $i$,  $\mathrm{char}(A_i)=p(x)(x-l_i)$ where $p(x)=(x^2-ax+\frac{1}{l_i})$ is irreducible over $\mathbb{F}_q$. 
	
	By Lemma \ref{-D2}, $(A_1, \ldots, A_r)$ is simultaneously diagonalizable to a tuple $D=(D_1,\ldots,D_r)\in \mathrm{SL}_3(\overline{\mathbb{F}}_q)^r$ of the following form up to simultaneous permutation, 
	$$\left(\begin{bmatrix}
		l_{1_1}& 0 & 0\\
		0 & l_{1_2} & 0\\
		0 & 0 & l_1
	\end{bmatrix}, \ldots ,\begin{bmatrix}
		l_{j_1}& 0 & 0\\
		0 & l_{j_2}& 0\\
		0 & 0 & l_{j}
	\end{bmatrix},\ldots ,\begin{bmatrix}
		l_{r_1}& 0 & 0\\
		0 & l_{r_2}& 0\\
		0 & 0 & l_{r}
	\end{bmatrix}  \right)$$ 
	where $l_j\in \mathbb{F}_q$ and $l_{j_1},l_{j_2} \in \overline{\mathbb{F}}_q- \mathbb{F}_q$ 
	or $l_{j_1}=l_{j_2}\in \mathbb{F}_q$ for $1\leq j \leq r$.
	Note here that we choose the tuple $(D_1,\ldots,D_r)$ with the basefield entry in the third row to represent the orbit of $(A_1,\ldots,A_r)$. For each $D_j$, the upper block $D'_j=\begin{pmatrix}
		l_{j_1} & 0 \\
		0 & l_{j_2}
	\end{pmatrix} \in \mathrm{GL}_2(\overline{\mathbb{F}}_q- \mathbb{F}_q) \cup Z(\mathrm{GL}_2(\mathbb{F}_q))$ which implies that $(D_1', \ldots, D_r') \in (\overline{\mathscr{D}}_2\cup Z)^r - Z^r$. The matrix $D'_j$ is unique up to permutation of entries for a fixed $D_j$. 
	This allows us to define a natural map between the two sets in consideration. 
	
	We first address the Weyl group action before defining the map.   Let $\mathscr{C}^r:=(\overline{\mathscr{D}}_2\cup Z)^r $  and $$\mathscr{W}_2 := \left\{ \begin{pmatrix}
		1&0\\0 &1
	\end{pmatrix},
	\begin{pmatrix}
		-1&0\\0 &-1
	\end{pmatrix},
	\begin{pmatrix}
		0&-1\\ 1 &0
	\end{pmatrix},
	\begin{pmatrix}
		0&1\\-1 &0
	\end{pmatrix}
	\right\}$$
	be the group in $\mathrm{SL}_2(\mathbb{F}_q)$ whose quotient is the Weyl group. 
	Note that the action of $\mathscr{W}_2$ on $D \in \mathrm{GL}_2(\overline{\mathbb{F}}_q)$ reduces to the simultaneous permutation of the rows of $D$. 
	 Let $G=\mathrm{SL}_3(\mathbb{F}_q)$ and $\mathscr{T}_{\overline{D}_{2}^r}/ G $ denote the orbits in the stratum. Now define: 
	\begin{eqnarray*}
		\Phi : 	\mathscr{T}_{\overline{D}_{2}^r}/ G & \longrightarrow &(\mathscr{C}^r- Z^r)/\mathscr{W}_2\\
		(D_1,\ldots,D_r) & \longmapsto & (D'_1,\ldots,D'_r)
	\end{eqnarray*}
	The map is clearly well defined. 
	Suppose $(D'_1,\ldots,D'_r)= (C'_1,\ldots,C'_r)$. Then, by definition $D'_j$ and $C'_j$ are simultaneous permutations of each other for all $j$ which implies that $(D_1,\ldots,D_r)$ and $(C_1,\ldots,C_r)$ are simultaneous permutations of each other. Therefore, the orbits of these elements are the same,  $[(D_1,\ldots,D_r)]=[(C_1,\ldots,C_r)]$ in $\mathscr{T}_{\overline{D}_{2}^r}$. Thus the map is injective. 
	\\To show surjectivity, let $(D'_1,\dots,D'_r)$ denote an element of $\mathscr{C}^r- Z(\mathrm{GL}_2(\mathbb{F}_q))^r$. Then, by Lemma  \ref{commutativity}, there exists a matrix $P'\in \mathrm{SL}_2(\overline{\mathbb{F}}_q)$ such that \\$(P'D'_1P^{'-1},\ldots,P'D'_rP^{'-1}) =(A'_1,\cdots,A'_r) \in \mathrm{GL}_2(\mathbb{F}_q)^r$. Now, consider the block matrices 
	
	$$(D_1,\ldots, D_r)= \left( \begin{pmatrix}
		D'_1 &0\\
		0 & \frac{1}{\det(D'_1)}
	\end{pmatrix}, \ldots,\begin{pmatrix}
		D'_r & 0\\
		0 & \frac{1}{\det(D'_r)}
	\end{pmatrix} \right).$$
	Clearly each $D_i \in \mathrm{SL}_3(\overline{\mathbb{F}}_q)$ by construction. Furthermore, define $P= \begin{pmatrix}
		P' &0\\
		0 &1 \end{pmatrix}$ and $A = \begin{pmatrix}
		A' &0\\
		0 & \frac{1}{\det(D'_r)} \end{pmatrix}$.
	Then $P\in \mathrm{GL}_3(\overline{\mathbb{F}}_q)$ and $\det(A)=1$ since $\det(A')=\det(D')$. It is easy to verify by block multiplication that $PD_i P^{-1}=A_i$. Therefore, $\phi([(A_1,\ldots,A_r)])=(D'_1,\ldots,D'_r)$ and the map is surjective. 
\end{proof}
\subsection{The case of $\mathrm{SL}_2$-matrices}

Let $\mathrm{char}(A)$ denote the characteristic polynomial of a matrix $A$. 
\begin{remark}\label{classify}
	The set $\mathscr{D}(\mathrm{SL}_2(\overline{\mathbb{F}}_q))^r$  can be classified as follows:    
	\begin{eqnarray*}		
		\text{ Reducible tuples: } 	\mathscr{D}_1^r(\mathrm{SL}_2(\mathbb{F}_q)) & := & \{(D_1,\ldots ,D_r) \ | \   \text{ there exists  } i \text{ such that } \\
		& & D_i \text { is not scalar, for } 1 \leq i\leq r \}\\
		\hspace{-1cm}\text{Central tuples: } \mathscr{D}_2^r(\mathrm{SL}_2(\mathbb{F}_q)) &:= & \{(D_1,\ldots ,D_r) \ | \  D_i \text { is scalar}  \text{ for } 1 \leq i\leq r \}	\\
		\text{ Irreducible tuples: }	\overline{\mathscr{D}}_2^r(\mathrm{SL}_2(\mathbb{F}_q)) & := & \{ (D_1,\ldots ,D_r) \ | \ \text{ there exists } i  \in \{1,\ldots, r\} \\  & & \text{ such that }  \mathrm{char}(D_i) \text{ is irreducible over } \mathbb{F}_q \}.
	\end{eqnarray*}
\end{remark}

We stratify the space $\mathrm{SL}_2(\mathbb{F}_q)^r$ as follows. 
\begin{definition}\label{n2stratum}
	A tuple $(A_1,\ldots,A_r)$ is said to be in $X$ stratum if there exists $P  \in \mathrm{GL}_2(\overline{\mathbb{F}}_q)$ such that $(PA_1P^{-1},\ldots,PA_rP^{-1})$ is an X tuple. Replace X with Reducible, Central and Irreducible to obtain the definition for the corresponding stratum. 
\end{definition}
In this subsection, we drop  $\mathrm{SL}_2(\mathbb{F}_q)$ from the subscript for convenience. The number of  orbits in a stratum $X$ under the conjugation action by $G$ is denoted by $\mathscr{T}_X/ G$. Since we are only considering the strata of diagonal tuples, each of these orbits are closed and hence constitute a point in the character variety. 

We use the following counts from [Theorem B, \cite{SCEpoly}] by Cavazos and Lawton. 
\begin{proposition}\label{thmSL2}
	Let $G=\mathrm{SL}_2(\mathbb{F}_q)$. The number of closed conjugation orbits in each stratum is given as follows: 
	\begin{enumerate}
		\item[1. ]	$	|\mathscr{T}_{\mathscr{D}_1^r} /G|   =   \begin{cases}
			2^r \text{ when } p \text{ is odd} \\
			1 \text{ when } p \text{ is even }
		\end{cases}$
		\item[2. ] $ |\mathscr{T}_{\mathscr{D}_2^r}/ G|  =  \begin{cases}
			\frac{(q-1)^r-2^r}{2} \text{ when } p \text{ is odd} \\
			\frac{(q-1)^r-1}{2} \text{ when } p \text{ is even }
		\end{cases}$ 
		\item[3. ] $	|\mathscr{T}_{\overline{\mathscr{D}}_2^r}/G|  =  \begin{cases}
			\frac{(q+1)^r-2^r}{2} \text{ when } p \text{ is odd} \\
			\frac{(q+1)^r-1}{2} \text{ when } p \text{ is even }.
		\end{cases}$
	\end{enumerate}
\end{proposition}

Recall from Theorem \ref{ThmKatz} [Section \ref{Epolynomial}] that the number of $\mathbb{F}_q$-points of the variety gives the $E$-polynomial. 
\begin{corollary}\label{corSL2E}
	The $E$-polynomial or Hodge-Deligne polynomial for the $\mathrm{SL}_2(\mathbb{C})$ character variety of the free abelian group, $\mathbb{Z}^r$, of rank $r$, denoted by $\mathfrak{X}_{\mathbb{Z}^r}(\mathrm{SL}_2(\mathbb{C}))$ is $\frac{(q+1)^r}{2}+\frac{(q-1)^r}{2}$. 
\end{corollary}
\begin{proof}
	Adding the counts from Proposition \ref{thmSL2} obtain the total number of conjugation orbits and thus the  $E$-polynomial of $\mathrm{SL}_2(\mathbb{C})$-character variety. When $p$ is odd,
	\begin{eqnarray*}
		\frac{(q-1)^r-2^r}{2}+ \frac{(q+1)^r-2^r}{2}+2^r=\frac{(q+1)^r}{2}+\frac{(q-1)^r}{2}.
	\end{eqnarray*}
	The case when $p$ is even yields the same result. 
\end{proof}

\subsection{The case of $\mathrm{SL}_3$-matrices}

Let $\mathscr{D}_1^r(\mathrm{SL}_3(\mathbb{F}_q)), \mathscr{D}_2^r(\mathrm{SL}_3(\mathbb{F}_q)),\mathscr{D}_3^r(\mathrm{SL}_3(\mathbb{F}_q))$, \\ $\overline{\mathscr{D}}_2^r(\mathrm{SL}_3(\mathbb{F}_q))$ and $\overline{\mathscr{D}}_3^r(\mathrm{SL}_3(\mathbb{F}_q))$ be the different types of tuples as defined in Definition \ref{Strata3}. Recall that for each of the above tuple $X$,  $\mathscr{T}_X$ denotes the corresponding strata. In this subsection, we omit  $\mathrm{SL}_3(\mathbb{F}_q)$ from the bracket for convenience. As in the case of $n=2$, let $\mathscr{T}_X/ G$ denotes the number of orbits in a stratum $X$ under the conjugation action by $G$. 
\begin{theorem} \label{n3}
	When $G=\mathrm{SL}_3(\mathbb{F}_q))$, the number of closed conjugation orbits in each stratum is as follows: 	
	\begin{enumerate}
		\item[1.]  Central Stratum $(\mathscr{T}_{\mathscr{D}_1^r}/ G)$
		$$ \begin{cases}
			3^r & \text{ if } p\equiv 1 \bmod 3  \text{ or }  p\equiv -1 \bmod 3 \text{ and } k \text { is even }\\
			1& \text { if } p\equiv 0 \bmod 3   \text { or }  p\equiv -1 \bmod 3 \text { and } k  \text{ is odd }.
		\end{cases}$$
		\item[2.] Repeating Stratum $(\mathscr{T}_{\mathscr{D}_2^r}/ G)$
		$$ \begin{cases}
			\left(q-1\right)^r-3^r	& \text{ if } p\equiv 1 \bmod 3  \text{ or }  p\equiv -1 \bmod 3 \text{ and } k \text { is even }\\
			\left(q-1\right)^r-1	&  \text { if } p\equiv 0 \bmod 3   \text { or }  p\equiv -1 \bmod 3 \text { and } k  \text{ is odd }.
		\end{cases}$$
		\item[3.] Reducible Stratum $(\mathscr{T}_{\mathscr{D}_3^r}/G)$
		$$ \begin{cases}
			\frac{\left(q-1\right)^{2r}}{6}\:-\frac{\left(q-1\right)^{r}}{2}+3^{r-1} & \text{ if } p\equiv 1 \bmod 3  \text{ or }  p\equiv -1 \bmod 3 \text{ and } k \text { is even }\\
			\frac{\left(q-1\right)^{2r}}{6}\:-\frac{\left(q-1\right)^r}{2}+\frac{1}{3} &  \text { if } p\equiv 0 \bmod 3   \text { or }  p\equiv -1 \bmod 3 \text { and } k  \text{ is odd }.
		\end{cases}$$
		
		\item[4.]  Irreducible Stratum $(\mathscr{T}_{\overline{\mathscr{D}}_3^r}/ G)$
		$$ \begin{cases}
			\frac{\left(q^2+q+1\right)^r}{3}-3^{r-1}	& \text{ if } p\equiv 1 \bmod 3  \text{ or }  p\equiv -1 \bmod 3 \text{ and } k \text { is even }\\
			\frac{\left(q^2+q+1\right)^r}{3}-\frac{1}{3}	&  \text { if } p\equiv 0 \bmod 3   \text { or }  p\equiv -1 \bmod 3 \text { and } k  \text{ is odd }.
		\end{cases}$$
		\item[5.]  Partially Reducible Stratum $(\mathscr{T}_{\overline{\mathscr{D}}_2^r}/ G)$
		$$ \begin{cases}
			\frac{\left(q^2-1\right)^r}{2}-\frac{\left(q-1\right)^r}{2} & \text{ if } p\equiv 1 \bmod 3  \text{ or }  p\equiv -1 \bmod 3 \text{ and } k \text { is even }\\
			\frac{\left(q^2-1\right)^r}{2}-\frac{\left(q-1\right)^r}{2}	&  \text { if } p\equiv 0 \bmod 3   \text { or }  p\equiv -1 \bmod 3 \text { and } k  \text{ is odd }.
		\end{cases}$$
	\end{enumerate}
\end{theorem} 
\begin{proof}
	Below, we give the proof for the case when $p \equiv 1 \bmod 3$ or $p\equiv -1 \bmod 3$ and $k$ is even. 
	\begin{enumerate}
		\item[1.] \textit{Central Stratum:  }By Proposition \ref{thm3.4.3}, number of central elements is $3$ (or $1$). Therefore, the result follows. 
		\item[2.] \textit{Repeating Stratum: }Let $(D_1,\ldots,D_r)$ represent a diagonal tuple in this stratum. Then, $D_i$ has exactly two repeated eigenvalues at fixed positions for all $1\leq i \leq r$ or $D_i$ is central. 
		From Proposition \ref{thm3.4.3}, there are $(q-4)$ (or $(q-2))$ elements with exactly two repeated eigenvalues. Along with the central elements, there are  $(q-1)$ possibilities for each $D_i$ and $(q-1)^r$ for the total number of such tuples. Finally, we subtract the central tuples which gives the desired count. Note that this is the number of such tuples up to Weyl group action. Since each such tuple can have $3$ distinct permutations by Lemma \ref{Rmk2}, the total number of tuples is $3((q-1)^r-3^r)$ (or $3((q-1)^r-1^r))$. 
		\item[3.] \textit{Reducible Stratum}: In this stratum, we count the number of diagonal tuples, $(D_1,\ldots,D_r)$ with entries from $\mathbb{F}_q$, excluding those included in the two strata mentioned above. This count is obtained by subtracting the counts of the previous two strata from the total count of diagonal matrices in $\mathrm{SL}_3(\mathbb{F}_q)$. 
		The number of diagonal matrices in  $\mathrm{SL}_3(\mathbb{F}_q)$, $|\mathscr{D}(\mathrm{SL}_3(\mathbb{F}_q))|= (q-1)^2$. Therefore, there are a total of $(q-1)^{2r}$ diagonal tuples in $\mathrm{SL}_3(\mathbb{F}_q)^r$. Subtracting the repeating and central strata and dividing by $6$ to account for the Weyl group action [Lemma \ref{Rmk2}], we get the following 
		\begin{eqnarray*}
			\frac{(q-1)^{2r}-3((q-1)^r-3^r)-3^r}{6} & = &\frac{(q-1)^{2r}}{6}-\frac{(q-1)^r}{2}+3^{r-1}.
		\end{eqnarray*}
		 Now we add the counts from previous two parts to obtain the total count for the basefield stratum, that is the number of orbits that are completely in the base field.
		\begin{eqnarray*}
			3^r+(q-1)^r-3^r+\frac{(q-1)^{2r}}{6}-\frac{(q-1)^r}{2}+3^{r-1} & = & \frac{(q-1)^{2r}}{6}+\frac{(q-1)^r}{2}+3^{r-1}.
		\end{eqnarray*}
		Replacing $3^r$ with $1^r$ obtains the count for the other case.
		\item[4.] \textit{ Irreducible stratum:}
		Let $(A_1,\ldots,A_r)$ be a tuple in the irreducible stratum. Recall that for any $A_i$, a commuting element $A_j\in \mathrm{SL}_3(\mathbb{F}_q)$ is an element of the stabilizer, $G_{A_i}\cap \mathrm{SL}_3(\mathbb{F}_q)$ under the conjugation action as explained in the proof of Lemma \ref{lem3.2.26}. By Lemma \ref{|G_A|}, $|G_{A_i}\cap \mathrm{SL}_3(\mathbb{F}_q)|=q^2+q+1$. Therefore, there are $q^2+q+1$ choices for each $A_i$. Additionally, from Lemma \ref{lem3.2.26}, we know that the number of distinct permutations possible for such a tuple is three. Since the stabilizer includes the elements from the central stratum, we can obtain the orbits in the irreducible stratum by discounting the central stratum: 
		\begin{eqnarray*}
			\frac{(q^2+q+1)^r-3^r}{3} &=& 	\frac{(q^2+q+1)^r}{3}-3^{r-1}.
		\end{eqnarray*}
		\item[5.] \textit{Partially reducible stratum }: 
		Recall that  $\overline{\mathscr{D}}(\mathrm{GL}_2(\overline{\mathbb{F}}_q -\mathbb{F}_q))$ denote the set of invertible $2 \times 2$ diagonal matrices with entries from $\overline{\mathbb{F}}_q-\mathbb{F}_q$.
		From Lemma \ref{PRS}, it suffices to compute $|\big(\overline{\mathscr{D}}(\mathrm{GL}_2(\overline{\mathbb{F}}_q -\mathbb{F}_q))\cup Z(\mathrm{GL}_2(\mathbb{F}_q))^r\big) - Z(\mathrm{GL}_2(\mathbb{F}_q))^r|$ up to the action of the Weyl group described in the proof of Lemma \ref{PRS}.
		For $D_i \in \overline{\mathscr{D}}(\mathrm{GL}_2(\overline{\mathbb{F}}_q- \mathbb{F}_q))$,  $\mathrm{char}(D_i)$ is an irreducible degree two polynomial. By Lemma \ref{lem3}, there are $\frac{q^2-q}{2}$ such polynomials and $|Z(\mathrm{GL}_2(\mathbb{F}_q))|=q-1$. Since there are two distinct diagonal matrices for each such polynomial, the total number of matrices in $\overline{\mathscr{D}}(\mathrm{GL}_2(\overline{\mathbb{F}}_q-\mathbb{F}_q))$ is $q^2-q$ and
		$$ |\overline{\mathscr{D}}(\mathrm{GL}_2(\overline{\mathbb{F}}_q-\mathbb{F}_q))\cup Z(\mathrm{GL}_2(\mathbb{F}_q))|  = (q^2-q+q-1)=q^2-1.$$
Thus, we obtain the number of distinct orbits in
		$\big(\overline{\mathscr{D}}(\mathrm{GL}_2(\overline{\mathbb{F}}_q- \mathbb{F}_q))\cup Z(\mathrm{GL}_2(\mathbb{F}_q))^r\big) - Z(\mathrm{GL}_2(\mathbb{F}_q))^r$ as $$\frac{(q^2-1)^{r}}{2}-\frac{(q-1)^r}{2}$$ after dividing by two to account for the Weyl group action.
	\end{enumerate}
\end{proof}
\begin{corollary}\label{3.2.29}\label{En3}
	The $E$-polynomial of the $\mathrm{SL}_3(\mathbb{C})$-character variety of $\mathbb{Z}^r$ is $$\frac{\left(q-1\right)^{2r}}{6}+\frac{\left(q^2-1\right)^r}{2}+\frac{\left(q^2+q+1\right)^r}{3}.$$
\end{corollary}
\begin{proof}
By Theorem \ref{ThmKatz}, the polynomial representing the number of points of the variety over $\mathbb{F}_q$ gives the $E$-polynomial of the variety. 
	We obtain the $E$-polynomial by adding the number of all the closed orbits over $\mathbb{F}_q$ from Theorem \ref{n3}. When $ p\equiv 1 \mod 3  \text{ or }  p\equiv -1 \mod 3 \text{ and } k \text { is even }$, we get
	$$	3^r+\left(q-1\right)^r-3^r+\frac{\left(q-1\right)^{2r}}{6}-\frac{\left(q-1\right)^r}{2}+3^{r-1}+\frac{\left(q^2-1\right)^r}{2}-\frac{\left(q-1\right)^r}{2}+\frac{\left(q^2+q+1\right)^r}{3}-3^{r-1}$$
	$$=\frac{\left(q-1\right)^{2r}}{6}+\frac{\left(q^2-1\right)^r}{2}+\frac{\left(q^2+q+1\right)^r}{3}.$$
	Similarly, adding the number of orbits  for the case when $ p\equiv 0 \mod 3   \text { or }  p\equiv -1 \mod 3 \text { and } k  \text{ is odd}$ yields the same result.  
\end{proof}
This agrees with the computation of E-polynomial, $e(M_1)$ in \cite{LawtonEpoly}. 
\section{ Asymptotic Transitivity}
 Character varieties of surface groups over fields of characteristic zero possess a well-defined geometric invariant measure. In this context, the notion of ergodicity describes how effectively a group action mixes the points within the variety. Specifically, the action of a group $G$ on a variety $X$ is said to be \textit{ergodic} if, for any $G$-invariant subset $A\subset X$, either $A$ or $X-A$ has measure zero. 
However, since there is no well-defined notion of a ``nice" invariant measure on the space of finite field points of the character varieties we study, we introduce the concept of asymptotic transitivity to capture the extent of ``mixing" under a group action.
\begin{definition}\label{AT}
	Let $G$ be a group and $X$ a variety defined over $\mathbb{Z}$. Suppose $G$ acts rationally on $X$. Now consider the action of $G$ on the $\mathbb{F}_q$-points of the variety denoted by $X(\mathbb{F}_q)$. \label{XFq} 
	We define the \textit{asymptotic ratio} as 
	\[  \lim_{q \to \infty}  \frac{\max\limits_{v\in X(\mathbb{F}_q)}\Big\lvert\mathrm{Orb}(v)\Big\rvert }{|X(\mathbb{F}_q)|}\] 
	whenever the limit exists. 
	\\We say that the action is \textit{asymptotically transitive} if the asymptotic ratio is one. 
\end{definition}
\begin{remark}
	\begin{enumerate}
		\item Since $q=p^n$, $q\to \infty$ in multiple ways, $p \to \infty$ or $n\to \infty$ or $p^n \to \infty$. We require that the limit exists regardless of how $q\to \infty$. 
		\item The ideal situation is when $|\mathrm{Orb}(v)|$ and $|X(\mathbb{F}_q)|$ are polynomials or quasi-polynomials in $q$. 
		\item To find the asymptotic ratio, we only need the formulas for $|\mathrm{Orb}(v)|$ and $|X(\mathbb{F}_q)|$ to be defined for all but finitely many values of $p$. 
	\end{enumerate}
\end{remark}
	Note that transitivity implies asymptotic transitivity, but the converse is not necessarily true. We are interested in the cases where the converse does not hold, that is, when the action is asymptotically transitive but not transitive. We now present a class of examples where the action is asymptotically transitive.
	\begin{lemma}
		Let $G$ be a group acting rationally on a variety $X$ defined over $\mathbb{Z}$.  Suppose $W(\mathbb{F}_q)$ is a subset of $X(\mathbb{F}_q)$ such that the action is transitive on $W(\mathbb{F}_q)$ for all $q$. For each $q$, let $\{V_i(\mathbb{F}_q)\}_{i=1}^m$ be a finite non-empty collection of disjoint subsets such that $W(\mathbb{F}_q) \not \subseteq V_i(\mathbb{F}_q)$  and $G\cdot V_i(\mathbb{F}_q)\subseteq V_i(\mathbb{F}_q)$ for $1\leq i\leq m$. 
		 If $W(\mathbb{F}_q)$ and $V_i(\mathbb{F}_q)$ have polynomial (monic) growth as $q$ increases such that the order of growth of $V_i({\mathbb{F}_q})$ is less than that of $W({\mathbb{F}_q})$ as $q\to \infty$ for $1\leq i \leq m$, then the action is asymptotically transitive on $Y(\mathbb{F}_q)$. 
	\end{lemma}
	\begin{proof}
		By assumption, the action is stable on $Y(\mathbb{F}_q)$. Clearly, $W(\mathbb{F}_q) \cap V_i(\mathbb{F}_q) = \emptyset$ since for $x \in W(\mathbb{F}_q) \cap V_i(\mathbb{F}_q)$, $\mathrm{Orb}(x) = W(\mathbb{F}_q)$ and $\mathrm{Orb}(x) \subseteq V_i(\mathbb{F}_q)$. This is a contradiction since $W(\mathbb{F}_q) \not\subseteq V_i(\mathbb{F}_q)$. Therefore, $|Y(\mathbb{F}_q)|=|W(\mathbb{F}_q)|+\sum\limits_{i=1}^m|V_i(\mathbb{F}_q)|$. 
		
		Suppose the order of growth of points in $W(\mathbb{F}_q)$ is $s$ and that of $V_i(\mathbb{F}_q)$ is $t_i$. By assumption, $s>t_1,\ldots,t_m$.  
		Let $w \in W(\mathbb{F}_q) \subset  Y(\mathbb{F}_q)$. Then the asymptotic ratio of the orbit of $w$ is given by 
		$$ \lim_{q\to \infty} \frac{|\mathrm{Orb}(w)|}{|Y(\mathbb{F}_q)|}=\lim_{q\to \infty} \frac{|W(\mathbb{F}_q)|}{|Y(\mathbb{F}_q)|}= \lim_{q\to \infty} \frac{q^s+a_1q^{s-1}+\cdots+a_{s}}{q^s+q^{t_1}+\cdots+q^{t_m}+\text{lower degree terms}}= 1$$
		since $q^s$ is the leading term of both numerator and denominator. 
	\end{proof}

The following is a specific example that demonstrates this.
\begin{example} 
	Let $\Gamma$ be a finitely presented group and $G(\mathbb{F}_q)$ be any group obtained as finite field points of an algebraic group $G$ defined over $\mathbb{Z}$. For example, $\mathrm{SL}_n(\mathbb{F}_q)$ denotes the $\mathbb{F}_q$-points of $\mathrm{SL}_n$. Consider the precomposition action of $\mathrm{Aut}(\Gamma)$ on $\mathrm{Hom}(\Gamma,G)$ by
	$ \alpha \cdot \rho = \rho \circ \alpha^{-1}$. Let $\rho\in \mathrm{Hom}(\Gamma,G)$ be such that $|\mathrm{Orb}(\rho)|\geq q$ for all $q\geq 2$. Clearly, $|\mathrm{Orb}(I)|=1$ where $I$ denotes the identity homomorphism. By letting $W(\mathbb{F}_q)=\mathrm{Orb}(\rho)$ and $V(\mathbb{F}_q)=\mathrm{Orb}(I)$, the action is asymptotically transitive on $Y(\mathbb{F}_q)=W(\mathbb{F}_q)\cup V(\mathbb{F}_q)$. 
\end{example}
\begin{remark}
	In general, we can choose $V_i(\mathbb{F}_q)$ to be the orbit of a point such that the orbit size is bounded by a constant and $W(\mathbb{F}_q)$ to be an orbit whose size is given by a monic polynomial in $q$ with degree at least $1$. 
\end{remark}
The automorphism group, $\mathrm{Aut}(\Gamma)$ acts naturally on $\mathrm{Hom}(\Gamma,G)$ by precomposition: 
$$ \alpha \cdot \rho = \rho \circ \alpha^{-1} \text{ where }  \alpha \in \mathrm{Aut}(\Gamma)   \text{ and } \rho \in \mathrm{Hom}(\Gamma,G).$$  We consider the case when $\Gamma=\mathbb{Z}^r$ and $G=\mathrm{SL}_n(\mathbb{F}_q)$ for $n=2,3$. 
Recall that $\mathrm{Aut}(\mathbb{Z}^r)=\mathrm{GL}_r(\mathbb{Z})$. Let $\gamma \in \mathrm{Aut}(\mathbb{Z}^r)$ and $A = [a_{ij}]
\in \mathrm{GL}_r(\mathbb{Z})$ corresponds to the automorphism, $\gamma^{-1}$ of $\mathbb{Z}^r$. 
For $(x_1,\ldots,x_r) \in \mathbb{Z}^r$ and $\rho\in \mathrm{Hom}(\mathbb{Z}^r,\mathrm{SL}_n\mathbb{F}_q))$, let $\rho(x_i)=Y_i\in \mathrm{SL}_n(\mathbb{F}_q)$. Then the action of $\Gamma$ on $\rho$ is defined as follows: 
\begin{eqnarray}\label{Eq2.1}
	\gamma \cdot \rho & = & \rho(\gamma^{-1}(x_1,\ldots,x_r))  \nonumber\\
	& = & \rho\left(\begin{pmatrix}
		a_{11} & \cdots & a_{1r}\\
		\vdots & \ddots & \vdots\\
		a_{r1} & \cdots & a_{rr}
	\end{pmatrix} \cdot \begin{pmatrix}
		x_1\\
		x_2 \\
		\vdots\\
		x_r
	\end{pmatrix}\right) \nonumber 
	= \rho \left(\begin{pmatrix}
		a_{11}x_1 +\cdots + a_{1r}x_r\\
		\vdots \\
		a_{r1}x_1 +\cdots + a_{rr}x_r
	\end{pmatrix}\right) \nonumber\\
	& = &\begin{pmatrix}
		a_{11}\rho(x_1) +\cdots + a_{1r}\rho(x_r)\\
		\vdots \\
		a_{r1}\rho(x_1) +\cdots + a_{rr}\rho(x_r)
	\end{pmatrix}  \nonumber 
	 =  \begin{pmatrix}
		Y_1^{a_{11}} \cdots Y_r^{a_{1r}}\\
		\vdots \\
		Y_1^{a_{r1}}\cdots Y_r^{a_{rr}}
	\end{pmatrix}. \nonumber
\end{eqnarray}

The action of $\mathrm{Aut}(\Gamma)$  induces a  natural action of the outer automorphism group, $\mathrm{Out}(\Gamma)$, on the character variety, $\mathfrak{X}_{\Gamma}(G)$.  In this section, we analyze the asymptotic dynamics of this action. We establish an upper bound for the asymptotic ratio in these cases and demonstrate that the action is not asymptotically transitive on the character variety. This is in contrast to the Markoff case as shown by Bourgain, Gamburd, and  Sarnak in \cite{Sarnak2} and \cite{Sarnak3}.

\begin{lemma}[Subgroup Lemma]\label{SL}
	Let $\Gamma=\langle \gamma_1,...,\gamma_r \ | \ R_i, \ i= 1,\ldots, s \rangle$ be a finitely presented group and $G$ be a group acting on $\mathrm{Hom}(\Gamma,G)$ as defined above. If $H$ is a subgroup of $G$ and $\rho\in \mathrm{Hom}(\Gamma,G)$ such that $(\rho(\gamma_1),\ldots,\rho(\gamma_r)) 
	\in H^r$, then $\mathrm{Orb}(\rho) \subseteq \mathrm{Hom}(\Gamma,H)$. 
\end{lemma}
\begin{proof}
	For $\rho \in \mathrm{Hom}(\Gamma,G))$, let $(\rho(\gamma_1),\ldots,\rho(\gamma_r))=(H_1,\ldots,H_r) \in H^r$. Let $\alpha \in \mathrm{Aut}(\Gamma)$ and $\alpha^{-1}(\gamma_i)= w_i \in \Gamma$ where $w_i= \gamma_{i_1}^{a_{i_1}}\cdots\gamma_{i_s}^{a_{i_s}}$. 
	Using this expression of $w_i$, we compute the following:
	\begin{eqnarray*}
		\rho(w_i) & = & \rho(\gamma_{i_1}^{a_{i_1}}\cdots\gamma_{i_s}^{a_{i_s}})
		=  \rho(\gamma_{i_1}^{a_{i_1}})\cdots \rho(\gamma_{i_s}^{a_{i_s}})
		=  H_{i_1}^{a_{i_1}}\cdots H_{i_s}^{a_{i_s}}.
	\end{eqnarray*}
	Therefore, $ \alpha \cdot \rho = (\rho(\alpha^{-1}(\gamma_1, \ldots,\gamma_r)))= (\rho(w_1), \ldots, \rho(w_r))$ \\$=(H_{1_1}^{a_{1_1}}\cdots H_{1_s}^{a_{1_s}},\ldots ,H_{i_1}^{a_{i_1}}\cdots H_{i_s}^{a_{i_s}},\ldots, H_{r_1}^{a_{r_1}}\cdots H_{r_s}^{a_{r_s}}).$
	Since $H_{i_j} \in H$ and $H$ is a subgroup, $\prod\limits_{j=1}^s H_{i_j}^{a_{i_j}}$ is also in $H$. Consequently, $\alpha \cdot \rho$ can be identified with an element in $H^r$. 
\end{proof}
\begin{remark}\label{RmkSL}
	Let $\Gamma=\mathbb{Z}^r$, $G=\mathrm{SL}_n({\mathbb{F}}_q)$ where $n=2,3$, and $H$ any subgroup of $G$. Suppose $\rho \in \mathrm{Hom}(\Gamma,G)$ be such that $\rho$ corresponds to a tuple $(A_1,\ldots,A_r) \in \mathrm{SL}_n({\mathbb{F}}_q)$. Then by Lemma \ref{SL}, if $A_i \in H$ for $1 \leq i \leq r$, then $\mathrm{Orb}(\rho)$ can be identified with a subset of $ H^r$. 
\end{remark}

\begin{theorem}\label{ATSL2}
	The action of $\mathrm{Out}(\mathbb{Z}^r)$ on the $\mathbb{F}_q$ points of the $\mathrm{SL}_n$-character variety of $\mathbb{Z}^r$ is not asymptotically transitive for $n=2,3$. Furthermore, the asymptotic ratio of the orbits of elements in the character variety is bounded above by $\frac{1}{2}$. 
\end{theorem}
\begin{proof}
We begin by proving the statement for $n=3$. 
The strategy is to first identify a subgroup within each stratum and then use the counts from the previous section to establish an upper bound for the asymptotic ratio.
	 Recall that the diagonal tuples (up to simultaneous permutation) correspond to the polystable points, which are the elements of the character variety, $\mathfrak{X}_{\mathbb{Z}^r}(\mathrm{SL}_3(\mathbb{F}_q))$. 
	Let $(A_1,\ldots,A_r)\in \mathrm{SL}_3(\mathbb{F}_q)^r$ be a tuple that is simultaneously diagonalizable to a tuple $(D_1,\ldots,D_r)\in \mathrm{SL}_3(\overline{\mathbb{F}}_q)^r$, unique up to Weyl group action. We use the equivalence class $[(D_1,\ldots,D_r)]$ to represent the corresponding element in $\mathfrak{X}_{\mathbb{Z}^r}(\mathrm{SL}_3(\mathbb{F}_q))$. 
		Recall the definition of tuples and strata from Definition \ref{Strata3}. Below, we classify the diagonal tuples into four distinct types and identify a subgroup within each type. In this proof, we use $a$ to represent the $r$-th power of the number of roots of unity for each $q$, i.e., $a=3^r$ when $p\equiv 1 \bmod 3 $ or $  p\equiv -1 \bmod 3$ and  $ k $ is even, and $a=1^r$ otherwise. For each of the following orbit types, we identify a subgroup, $H$ such that $A_i$ is an element of $H$. By Lemma \ref{SL},  the size of $\mathrm{Orb}((A_1,A_2,\ldots,A_r))$ is bounded by $|H^r|$. Consequently, the maximal size of $\mathrm{Orb}([(A_1,A_2,\ldots,A_r)])$ is bounded above by the number of elements in each stratum (up to a constant). This allows us to compute a bound for the asymptotic ratio of each type. 
	\begin{enumerate}
		\item Basefield Subgroup Tuples \\ This includes the set of all tuples whose entries come from the set of diagonal elements in $\mathrm{SL}_3(\mathbb{F}_q)$ i.e.,  $(D_1,\ldots,D_r) \in \mathscr{D}(\mathrm{SL}_3(\mathbb{F}_q))^r$. The set, $\mathscr{D}(\mathrm{SL}_3(\mathbb{F}_q))$ is a subgroup of $\mathrm{SL}_3(\mathrm{F}_q)$. Note that this case represents the union of reducible, repeating, and central tuples, as defined in Definition \ref{Strata3}. Therefore, the total number of elements in $\mathfrak{X}_{\mathbb{Z}^r}(\mathrm{SL}_3(\mathbb{F}_q))$ of this type can be calculated as follows by summing the number of points in these three strata. Recall that $q=p^k$.  From Theorem  \ref{n3}, we obtain
		\begin{eqnarray*}
			a+(q-1)^r-a+\frac{(q-1)^{2r}}{6}-\frac{(q-1)^r}{2}+\frac{a}{3} & = &  \frac{(q-1)^{2r}}{6}+\frac{(q-1)^r}{2}+\frac{a}{3}.
		\end{eqnarray*} 
	Therefore, for any tuple, $(D_1,\ldots,D_r) \in \mathscr{D}(\mathrm{SL}_3(\mathbb{F}_q))^r$, the size of the orbit,  $|\mathrm{Orb}([(D_1,\ldots,D_r)])|$,  is less than the above count. 
		\item Central Tuples \\
		The central matrices, $Z(\mathrm{SL}_3(\mathbb{F}_q))$ form a subgroup of $\mathrm{SL}_3(\mathbb{F}_q)$. This subgroup is included in the previous type. Although distinguishing this type is not necessary for computational purposes, we make this distinction to highlight that the central tuples are fixed points under this action. The number of central tuples is $a$. Hence, the size of the orbit of any central tuple is less than $a$. However, note that the central tuples do not give rise to a maximal orbit. 
		\item Irreducible Subgroup Tuples \\
		Let $[(A_1,\ldots,A_r)]$ be in the irreducible stratum. By definition, $A_i$ has an irreducible characteristic polynomial for some $i$. Then, $A_j \in G_{A_i}\cap \mathrm{SL}_3(\mathbb{F}_q)$, the stabilizer subgroup of $A_i$ [see the proof of Lemma \ref{lem3.2.26}]. Since $G_{A_i}$ is a subgroup, by Lemma \ref{SL},  $\mathrm{Orb}([(A_1,\ldots,A_r)])$ is contained in the irreducible stratum. By Theorem \ref{n3}, number of elements in the irreducible stratum is  
		$$ \frac{\left(q^2+q+1\right)^{r}-a}{3},$$ 
		which bounds the size of any orbit in the irreducible stratum. 
		\item Partially Reducible Tuples \\ 
		 Let $(A_1,\ldots,A_r)$ be an element of the partially reducible stratum. Then $D_i$ is either central or has the property that $\mathrm{char}(D_i)$ has an an irreducible quadratic factor. As explained in the proof of Lemma \ref{PRS}, the set of such matrices can be identified with  $\mathscr{D}(\mathrm{GL}_2(\overline{\mathbb{F}}_q-\mathbb{F}_q)) \cup Z(\mathrm{GL}_2(\mathbb{F}_q))$. This set, along with the central elements, has a group structure as it is the stabilizer subgroup of the matrices in the partially reducible stratum as explained in the previous case. From Theorem $\ref{n3}$, we have the following count for the number of points in this stratum 
$$\frac{\left(q^2+q+1\right)^{r}-a}{3}$$
which provides an upper bound for the size of any orbit in this stratum. 
	\end{enumerate}
Now that we have counts for the maximal orbit size of each type, we can compute a bound for the asymptotic ratio using the count for the total number of points in  $\mathfrak{X}_{\mathbb{Z}^r}(\mathrm{SL}_3(\mathbb{F}_q))$ from Corollary \ref{3.2.29}:  $$\frac{\left(q-1\right)^{2r}}{6}+\frac{\left(q^2-1\right)^r}{2}+\frac{\left(q^2+q+1\right)^r}{3}.$$
	Note that the number of points in the variety is a monic polynomial of degree $2r$, i.e, the coefficient of $q^{2r}$ is one. 
	\begin{enumerate}
		\item Basefield Orbit \\
		$$	\lim_{q\to \infty}	\frac{\frac{\left(q-1\right)^{2r}}{6}+\frac{\left(q-1\right)^r}{2}+\frac{a}{3}}{\frac{\left(q-1\right)^{2r}}{6}+\frac{\left(q^2-1\right)^r}{2}+\frac{\left(q^2+q+1\right)^r}{3}} =\frac{1}{6} 
		$$
		The orbit grows at a rate of $q^{2r}$ with a coefficient of $\frac{1}{6}$. The limit of the ratio is determined by the coefficient of $q^{2r}$ in the numerator. 
		\item Central Orbit 
		$$
		\lim_{q\to \infty}	\frac{a}{\frac{\left(q-1\right)^{2r}}{6}+\frac{\left(q^2-1\right)^r}{2}+\frac{\left(q^2+q+1\right)^r}{3}} 
		=0$$
		If all elements in the tuple are central, the size of the orbit is bounded by a constant, and thus the asymptotic ratio is zero.
		\item Partially Reducible Orbit
	$$
			\lim_{q\to \infty}	\frac{\frac{(q^2-1)^{r}-(q-1)^r}{2}+a}{\frac{\left(q-1\right)^{2r}}{6}+\frac{\left(q^2-1\right)^r}{2}+\frac{\left(q^2+q+1\right)^r}{3}}=\frac{1}{2} 
	$$
	In this case, the coefficient of the leading term $q^{2r}$ is $\frac{1}{2}$. 
		\item  Irreducible Orbit
		$$
			\lim_{q\to \infty} \frac{\frac{\left(q^2+q+1\right)^{r}-a}{3}}{\frac{\left(q-1\right)^{2r}}{6}+\frac{\left(q^2-1\right)^r}{2}+\frac{\left(q^2+q+1\right)^r}{3}}= \frac{1}{3} 			
	$$
		\end{enumerate}
Based on the above calculations, it follows that the the asymptotic ratio of orbits under the action of $\mathrm{Out}(\mathbb{Z}^r)$ cannot exceed $\frac{1}{2}$. In particular, the ratio cannot equal one, so the action is not asymptotically transitive.
\\	For $n=2$, the proof follows in a similar manner with three orbit types.  Recall the definition of strata from Remark \ref{classify} and Definition \ref{n2stratum}. In this case, the strata and the orbit types coincide.  The subgroups in each stratum are identical to those in the case when $n=3$, which we outline below. Note that each orbit type includes the central tuples. In the following counts, let $b=2^r$ when$p$ is odd, and $b=1^r$ when $p=2$. 

	\begin{enumerate}
		\item Central Tuples \\ 
		The set of central elements, $Z(\mathrm{SL}_2(\mathbb{F}_q))$, is a subgroup of $\mathrm{SL}_2(\mathbb{F}_q)$. The number of elements in this stratum is $b$. 
		\item Reducible Tuples \\
		The set of diagonal elements, $\mathscr{D}(\mathrm{SL}_2(\mathbb{F}_q))$,  is a subgroup of $\mathrm{SL}_2(\mathbb{F}_q)$. The number of tuples in this stratum, along with the central tuples, can be computed from Proposition  \ref{thmSL2} as follows: 
		$$ \frac{(q-1)^r-b}{2}+b =\frac{(q-1)^r+b}{2}$$ 
		which provides an upper bound for the size of the orbit of any element in this stratum. 
		\item Irreducible: The set of irreducible diagonal tuples, $\mathscr{D}_{irr}(\mathrm{SL}_2(\mathbb{F}_q))$ along with the central elements is the stabilizer subgroup, $G_A$, of a diagonal matrix $A$ such that $\mathrm{char}(A)$ is irreducible (as explained in the case when $n=3$). 
	The size of any orbit in this stratum is bounded above by the sum of the points in the stratum and the central tuples which can be counted as follows from \ref{thmSL2}: 
		$$	\frac{(q+1)^r-b}{2}+b=\frac{(q+1)^r+b}{2}.$$ 
\end{enumerate}	
We compute an upper bound for the asymptotic ratios of orbits in each  stratum using the counts from Proposition \ref{thmSL2}. From Corollary \ref{corSL2E}, we obtain the number of points in the character variety, $\mathfrak{X}_{\mathbb{Z}^r}(\mathrm{SL}_2(\mathbb{F}_q))$ as $$\frac{(q+1)^r+(q-1)^r}{2}.$$ 
Note that the leading coefficient here is $q^r$. 
 
\begin{enumerate} 
	\item Central Orbit:   $$
		\lim_{q\to \infty}	\frac{b}{ \frac{(q+1)^r}{2}+\frac{(q-1)^r}{2}}  
		=0$$
	\item Reducible Orbit: 
	$$
		\lim_{q\to \infty}	\frac{\frac{(q-1)^r+b}{2}}{\frac{(q+1)^r}{2}+\frac{(q-1)^r}{2}} = 
		\frac{1}{2}$$
	\item Irreducible Orbit: 
	$$ \lim_{q\to \infty}	\frac{\frac{(q+1)^r+b}{2}}{ \frac{(q+1)^r}{2}+\frac{(q-1)^r}{2}}
		= \frac{1}{2}.$$
\end{enumerate} 
Therefore, the asymptotic ratio of the $\mathrm{Out}(\mathbb{Z}^r)$-action on  $\mathfrak{X}_{\mathbb{Z}^r}(\mathrm{SL}_2(\mathbb{F}_q))$ is bounded above by $\frac{1}{2}$, implying that the action is not asymptotically transitive. 
\end{proof}

An interesting question to ask is whether this holds for general $\mathrm{SL}_n$-character varieties of $\mathbb{Z}^r$ for $n\geq 4$. We believe the results on asymptotic transitivity extends to $\mathrm{SL}_n$-character varieties for $n\geq 4$. It will also be interesting to explore asymptotic transitivity of $\mathrm{SL}_n$-character varieties of free groups. 
The observations from Theorem \ref{ATSL2} give us reason to believe that $\mathrm{SL}_n$-character varieties of free abelian groups or even free groups are not good candidates for asymptotic transitivity. However, there exist certain relative character varieties for which the action is asymptotically transitive. Sarnak and collaborators showed that the affine varieties arising from certain relative character varieties are asymptotically transitive, see \cite{Sarnak2} and \cite{Sarnak3}. Moreover, in the Mason Experimental Geometry Lab (MEGL) \cite{MEGL2023}, a group of students identified a non-trivial action on a linear projective space that is not transitive but shows asymptotic transitivity. 

One potential application arises from the asymptotic ratio which helps  in estimating the $E$-polynomial.  Recall that the asymptotic ratio is defined as the limit of the ratio of the size of the largest orbit to that of the number of points in the variety as $q$ approaches infinity.  Since the $E$-polynomial gives the number of points in a variety, if $X$ is a variety that has polynomial count equipped with a fixed asymptotic ratio for a defined group action, then we can estimate the $E$-polynomial using the size of the largest orbit. Moreover, if the $E$-polynomial is known, then the asymptotic ratio allows us to estimate the size of the maximal orbit. For instance, in the $\mathrm{SL}_3$-character variety of $\mathbb{Z}^r$ under the action of $\mathrm{Out}(\mathbb{Z}^r)$, the size of the maximal orbit cannot exceed half of the size of the variety for large $q$. Recall from Section \ref{Epolynomial} that letting $q=1$ in the $E$-polynomial gives us the Euler characteristic. Consequently, an estimate of the $E$-polynomial offers insight into approximating the Euler characteristic of the variety.


\end{document}